\documentclass[12 pt]{amsart}
\usepackage{fullpage} 
\usepackage{hyperref}
\usepackage{etex}
\usepackage[shortlabels]{enumitem}
\usepackage{amsmath}
\usepackage{amsxtra}
\usepackage{amscd}
\usepackage{amsthm}
\usepackage{amsfonts}
\usepackage{amssymb}
\usepackage{eucal}
\usepackage[all]{xy}
\usepackage{graphicx}
\usepackage{tikz-cd}
\usepackage{mathrsfs}
\usepackage{subfiles}
\usepackage{mathpazo}
\usepackage[colorinlistoftodos, textsize=tiny]{todonotes}
\usepackage{morefloats}
\usepackage{pdfpages}
\usepackage{thm-restate}
\usepackage[utf8]{inputenc}
\usepackage[T1]{fontenc}
\usepackage{epigraph}
\usepackage{csquotes}
\usepackage{moreverb}
\usepackage[margin=1in]{geometry}
\usepackage{tikz}
\graphicspath{ {images/} }

\RequirePackage{color}
\definecolor{myred}{rgb}{0.75,0,0}
\definecolor{mygreen}{rgb}{0,0.5,0}
\definecolor{myblue}{rgb}{0,0,0.65}

\usepackage{hyperref}
\hypersetup{citecolor=blue}
\usepackage{tikz}
\usetikzlibrary{matrix,arrows,decorations.pathmorphing}

\theoremstyle{plain}
\newtheorem{theorem}{Theorem}[section]
\newtheorem{proposition}[theorem]{Proposition}
\newtheorem{lemma}[theorem]{Lemma}
\newtheorem{corollary}[theorem]{Corollary}
\theoremstyle{definition}
\newtheorem{definition}[theorem]{Definition}
\newtheorem{remark}[theorem]{Remark}
\newtheorem{example}[theorem]{Example}

\newtheorem{conjecture}[theorem]{Conjecture}

\theoremstyle{remark}
\newtheorem{notation}[theorem]{Notation}
\numberwithin{equation}{section}
  
\newcommand\nc{\newcommand}
\nc\on{\operatorname}
\nc\renc{\renewcommand}

\newcommand\bc{{\mathbb C}}
\newcommand\br{{\mathbb R}}
\newcommand\bq{{\mathbb Q}}
\newcommand\bp{{\mathbb P}}

\newcommand\bz{{\mathbb Z}}

\newcommand\bg{{\mathbb G}}

\newcommand\sce{\mathscr E}

\newcommand\sck{\mathscr K}

\newcommand\sco{\mathscr O}

\newcommand\scx{\mathscr X}

\newcommand \ra{\rightarrow}
\newcommand \xra{\xrightarrow}
\DeclareMathOperator\spec{\text{Spec}}
\DeclareMathOperator\proj{\text{Proj}}

\newcommand*{\shom}{\mathscr{H}\kern -.5pt om}
\newcommand*{\stor}{\mathscr{T}\kern -.5pt or}
\newcommand*{\sext}{\mathscr{E}\kern -.5pt xt}

\makeatletter
\newcommand{\customlabel}[2]{\protected@write \@auxout {}{\string \newlabel {#1}{{#2}{\thepage}{#2}{#1}{}} }\hypertarget{#1}{#2}}

\newcommand\aff[2]{\mathbb A_{#2}^{12{#1}+3}} 
\newcommand\espace[2]{\mathscr W_{#2}^{#1}} 
\newcommand\estack[2]{\underline{\mathscr W}_{#2}^{#1}} 
\newcommand\smestack[2]{{\underline{\mathscr W}^{\circ}}_{#2}^{#1}} 
\newcommand\smespace[2]{\mathrm{\mathscr W^{\circ}}_{#2}^{#1}} 
\newcommand\uespace[2]{\mathscr U\!\!\mathscr W_{#2}^{#1}} 
\newcommand\usmespace[2]{\mathrm{\mathscr U\!\!\mathscr W^{\circ}}_{#2}^{#1}} 
\newcommand\sspace[3]{\mathrm{Sel}_{#1, #3}^{#2}} 
\newcommand\sstack[3]{\underline{\mathrm{Sel}}_{#1, #3}^{#2}} 
\newcommand\smsstack[3]{{\underline{\mathrm{Sel}}^{\circ}}_{#1, #3}^{#2}} 
\newcommand\smsspace[3]{{\mathrm{Sel}^{\circ}}_{#1, #3}^{#2}} 
\newcommand\ssheaf[3]{\mathcal{S}e\ell_{#1, #3}^{#2}} 
\newcommand\smssheaf[3]{{\mathcal{S}e\ell^{\circ}}_{#1, #3}^{#2}} 
\newcommand\comp[2]{\Phi_{#1, #2}} 
\newcommand\fcomp[1]{\Phi_{#1}} 
\newcommand\mono[3]{\rho_{#1, #3}^{#2}} 
\newcommand\mgeom[3]{\im \rho_{#1, \ol{#3}}^{#2}} 
\newcommand\places[1]{\Sigma_{#1}}

\newcommand\sh{\mathrm{sh}}

\newcommand\vs[2]{V_{#1,#2}} 
\newcommand\vsel[3]{V_{#1}^{#2}} 
\newcommand\qsel[3]{Q_{#1}^{#2}} 
\newcommand\divsing[2]{\mathcal D_{\mathrm{sing},#2}^{#1}}
\newcommand\divadd[2]{\mathcal D_{\mathrm{cusp},#2}^{#1}}

\DeclareMathOperator\id{id}

\DeclareMathOperator\coker{coker}

\DeclareMathOperator\rk{rk}

\DeclareMathOperator\pic{Pic}

\DeclareMathOperator\im{im}

\DeclareMathOperator\sym{Sym}
\DeclareMathOperator\pgl{PGL}

\DeclareMathOperator\speci{sp}

\DeclareMathOperator\chr{char}

\newcommand\ol{\overline}

\DeclareMathOperator\aut{Aut}
\DeclareMathOperator\gl{GL}

\renewcommand\o{{\rm{O}}}
\newcommand\osp{{\rm{O^*_{-1}}}}
\DeclareMathOperator\sel{Sel}

\DeclareMathOperator\val{val}

\DeclareMathOperator\sm{sm}

\DeclareMathOperator\disc{disc}
\DeclareMathOperator\average{Average}

\DeclareMathOperator\grp{grp}
\DeclareMathOperator\et{\acute et}
\DeclareMathOperator\tame{tame}

\renewcommand\top{\mathrm{top}}
\DeclareFontFamily{U}{wncy}{}
\DeclareFontShape{U}{wncy}{m}{n}{<->wncyr10}{}
\DeclareSymbolFont{mcy}{U}{wncy}{m}{n}
\DeclareMathSymbol{\Sha}{\mathord}{mcy}{"58}

\def\listtodoname{List of Todos}
\def\listoftodos{\@starttoc{tdo}\listtodoname}

\title{The
geometric 
average size of Selmer groups over function fields}
\author{Aaron Landesman} \address{Department of Mathematics, Stanford University, \mbox{Stanford, CA 94305}}
\email{aaronlandesman@stanford.edu}

\usepackage{microtype}
\begin{document}

\begin{abstract}
We show, in the large $q$ limit, that the average size of $n$-Selmer groups of elliptic curves of bounded height over $\mathbb F_q(t)$
is the sum of the divisors of $n$.
As a corollary, 
again in the large $q$ limit, we deduce that $100\%$ of elliptic curves of bounded height over $\mathbb F_q(t)$ have rank $0$ or $1$.
\end{abstract}

\maketitle

\section{Introduction}

One recent goal in arithmetic statistics is to determine the average size of $n$-Selmer
groups. 
The majority of results in this direction have concentrated on the regime $n \leq 5$. 
The goal of this paper is to
describe the average size of $n$-Selmer groups, in the large $q$ limit,
for arbitrary $n$, over function fields of the form $\mathbb F_q(t)$.
To start, we recall the following conjecture on the average size of $n$-Selmer groups.

\begin{conjecture}[Bhargava--Shankar \protect{\cite[Conjecture 4]{bhargavaS:average-4-selmer}}, Poonen--Rains \protect{\cite[Conjecture 1.4(b)]{poonenR:random-maximal-isotropic-subspaces-and-selmer-groups}}, and Bhargava--Kane--Lenstra--Poonen--Rains \protect{\cite[\S 5.7]{bhargavaKLPR:modeling-the-distribution-of-ranks-selmer-groups}}]
	\label{conjecture:bhargava-shankar}
	When all elliptic curves are ordered by height,
the average size of $n$-Selmer groups is the sum of the divisors of $n$.
\end{conjecture}

So far, \autoref{conjecture:bhargava-shankar}
has been proven for $2, 3, 4$, and $5$-Selmer
groups over $\bq$ by Bhargava and Shankar \cite{bhargava-shankar:binary-quartic-forms-having-bounded-invariants, bhargavaS:ternary, bhargavaS:average-4-selmer, bhargavaS:average-5-selmer},
for $3$-Selmer groups over $\mathbb F_q(t)$ by de Jong \cite{de-jong:counting-elliptic-surfaces-over-finite-fields},
and for $2$-Selmer groups over function fields by 
H{\`{\^o}}, L\^e H\`ung, and Ng{\^o}
\cite{ho-lehung-ngo:average-size-of-2-selmber-groups}.
These proofs all depend on specific geometric descriptions of $n$-Selmer elements
for $n \leq 5$ and seem difficult to generalize to 
$n > 5$.
There have also been predictions for the full distributions of $n$-Selmer groups in
\cite{poonenR:random-maximal-isotropic-subspaces-and-selmer-groups} and
\cite{bhargavaKLPR:modeling-the-distribution-of-ranks-selmer-groups}.

We now restrict our attention to global fields of the form 
$\mathbb F_q(t)$ with $\chr(\mathbb F_q) \neq 2$.
For any elliptic curve $E$ over such a field,
there is a unique $d$ 
so that $E$
can be written in minimal Weierstrass form as
$y^2z = x^3 + a_2(s,t)x^2z + a_4(s,t)xz^2 + a_6(s,t)z^3$,
where $a_{2i}(s,t) \in \mathbb F_q[s,t]$ is a homogeneous polynomial of degree $2id$, for $i \in \left\{ 1,2,3 \right\}$
(see \autoref{subsubsection:height}).
Define the {\em height} of $E$, notated $h(E),$ to be this value of $d$.

We next define several notions of the average size of the $n$-Selmer group.
Let $T$ be a function sending isomorphism classes of elliptic curves over $\mathbb F_q(t)$ to $\br$. 
Letting 
$\sel_n(E)$ denote the
$n$-Selmer group of $E$,
one relevant such statistic is $\# \sel_n$, given by $E \mapsto \# \sel_n(E)$. 
Letting $E$ range over isomorphism classes of elliptic curves over $\mathbb F_q(t)$, the {\em average size} of $T$ over $\mathbb F_q(t)$ for elliptic curves of height up to $d$, if it exists, is then
	\begin{equation}
		\label{equation:average}
		\average^{\leq d}(T/\mathbb F_q(t)) := \frac{\sum\limits_{\substack{E/\mathbb F_q(t), \, h(E) \leq d}} T(E)}{\# \left\{ E: E/\mathbb F_q(t), \, h(E) \leq d \right\}}.
\end{equation}
In \autoref{conjecture:bhargava-shankar}, ``the average size of the $n$-Selmer groups'' over $\mathbb F_q(t)$ can be formulated as 
$\lim_{d \ra \infty} \average^{\leq d}(\# \sel_n /\mathbb F_q(t))$.
While this may be challenging to compute for general $n$, we can try to determine this by first taking the limit in $q$, and then
taking the limit in $d$.
For $q$ ranging over prime powers, we deem $\lim\limits_{\substack{q \ra \infty\\ \gcd(q,2n)=1}} \average^{\leq d}(\# \sel_n / \mathbb F_q(t))$ the {\em geometric average size} of the $n$-Selmer group of height up to $d$.
Our main result is that this modified limit evaluates to the prediction of \autoref{conjecture:bhargava-shankar} for $d \geq 2$.
\begin{theorem}
	\label{theorem:number-components}
	For $n \geq 1$ and $d \geq 2$,
	\begin{align}
		\label{equation:average-selmer}
		\lim_{\substack{q \ra \infty \\ \gcd(q,2n) = 1}}
		\average^{\leq d}(\# \sel_n / \mathbb F_q(t))
		=
		\sum_{m \mid n} m.
	\end{align}
\end{theorem}

In \autoref{theorem:number-components}, the limit is taken over all prime powers $q$ relatively prime to $2n$.
In particular, the same statement holds when one restricts $q$ to range over powers of a fixed prime $p$ that does not divide $2n$.
See \autoref{subsection:remarks} for further related remarks.

We next state a corollary showing that, in the large $q$ limit, $100\%$ of elliptic curves of height up to $d$ either have rank $0$ or $1$.
In subsequent work
\cite[Corollary 1.5]{fengLR:geometric-distribution-of-selmer-groups}, we were able to strengthen \autoref{corollary:average-rank} to compute the distribution
of ranks of elliptic curves in this large $q$ limit.
As we learned in private communication with Manjul Bhargava,
the following line of reasoning is based upon an idea he had many years ago, and is explained in
\cite[Proposition 5]{bhargavaS:average-4-selmer} and
\cite[p.246-247]{poonenR:random-maximal-isotropic-subspaces-and-selmer-groups}.
As we are taking a limit in $q$ instead of $d$, the situation is slightly different from that in the above references,
and we now spell out the details.
To make a precise statement, define the statistic $\delta_{\rk \geq 2}$ on isomorphism classes of elliptic curves $E$ by 
$\delta_{\rk \geq 2}(E) = 1$ if $\rk(E) \geq 2$ and $\delta_{\rk \geq 2}(E) =0$ if $\rk(E) < 2$.
In this case, $\average^{\leq d}(\delta_{\rk \geq 2}/\mathbb F_q(t))$ is the proportion of isomorphism classes of elliptic curves over $\mathbb F_q(t)$ of height up to $d$ with rank
at least $2$.
The following corollary can alternatively be deduced from \cite[Theorem 13.3.3]{katz:moments-monodromy-and-perversity}.
\begin{corollary}
	\label{corollary:average-rank}
	For $d \geq 2$, $\lim\limits_{\substack{q \ra \infty\\ \gcd(q, 2)=1}} \average^{\leq d}(\delta_{\rk \geq 2} / \mathbb F_q(t)) = 0$.
\end{corollary}
\begin{proof}
	Suppose 
	$\lim\limits_{\substack{q \ra \infty\\ \gcd(q, 2)=1}} \average^{\leq d}(\delta_{\rk \geq 2} / \mathbb F_q(t)) \neq 0$.
	Then, there is some $\varepsilon > 0$ and some infinite sequence of powers of odd primes $\{q_i\}_{i \in \bz_{> 0}}$
	so that $\average^{\leq d}(\delta_{\rk \geq 2} / \mathbb F_{q_i}(t)) > \varepsilon$ for all such $q_i$.
	Choose a sufficiently large prime number $n$ so that
$\frac{n+2}{n^2} < \varepsilon$
and $n$ is relatively prime to infinitely many of the $q_i$.
Replace the sequence $\{q_i\}_{i \in \bz_{> 0}}$ by an infinite subsequence with all terms relatively prime to $n$.
	Since $n^2 \delta_{\rk \leq 2}(E) \leq n^{\rk(E)} \leq \# \sel_n(E)$ for any elliptic curve $E$, we find
	that for any $q_i$ in the above sequence,
	\begin{align*}
	n^2 \varepsilon < n^2 \average^{\leq d}(\delta_{\rk \geq 2} / \mathbb F_{q_i}(t)) \leq \average^{\leq d}\left( \# \sel_n / \mathbb F_{q_i}(t) \right).
	\end{align*}
However, as $n$ is prime and $d \geq 2$, by \autoref{theorem:number-components}, 
$\lim\limits_{\substack{q \ra \infty\\ \gcd(q, 2n)=1}}
\average^{\leq d}( \#\sel_n / \mathbb F_{q}(t)) = n +1$.
Therefore, there can only finitely many $q_i$ which are relatively prime to $2n$ so that $\average^{\leq d}( \#\sel_n / \mathbb F_{q_i}(t)) \geq n + 2$.
Hence, $\varepsilon n^2 < n + 2$, contradicting our choice of $n$.
\end{proof}

\subsection{Remarks on \autoref{theorem:number-components}}
\label{subsection:remarks}

We now make some remarks on various aspects of \autoref{theorem:number-components} including heuristics, homological stability, families of quadratic twists, and more.
\begin{remark}[Four heuristics for the average size of Selmer groups]
	\label{remark:three-heuristics-average}
	When computing the average size of Selmer groups, it is natural to ask if there is some deeper reason for
	why the average size of $n$-Selmer groups should be $\sum_{m \mid n} m$. Here are four heuristics which suggest this description.
	\begin{enumerate}
		\item 	In 
\cite[Conjecture 4]{bhargavaS:average-4-selmer},
the average size of $n$-Selmer groups is related to the fact that the Tamagawa number $\tau(\pgl_m)$ is $m$, and their average size is
the sum $\sum_{m \mid n} \tau(\pgl_m)$ for $m \mid n$.
	The same heuristic is used in
	\cite{bhargava-shankar:binary-quartic-forms-having-bounded-invariants, bhargavaS:ternary, bhargavaS:average-5-selmer},
and \cite{ho-lehung-ngo:average-size-of-2-selmber-groups}.
\item In this paper, we present another heuristic: the average size of $n$-Selmer groups is the number of orbits
	of a certain orthogonal group $\o(\qsel n d k)$ on a rank $12d-4$ free $\bz/n\bz$ module.
	Such orbits are in bijection with irreducible 
	components of a moduli space for Selmer elements, which we call the
	$n$-Selmer space in \autoref{definition:selmer-space}.

\item A third heuristic appears in \cite{de-jong:counting-elliptic-surfaces-over-finite-fields}
	for $3$-Selmer groups, in \cite{vakil:twelve-points-on-the-projective-line} for $2$-Selmer groups, and in
	\cite[Theorem 5.4]{de-jongF:on-the-geometry-of-principal-homogeneous-spaces} for $n$-Selmer groups.
	These works suggest that the average size of the $n$-Selmer group should equal 
	the number of balanced (also called rigid) rank $m$ projective bundles over $\bp^1$ 
	for $m \mid n$. 
	Indeed, the balanced rank $m$ projective bundles are all of the form $\proj_{\bp^1}\sym^\bullet\left( \sco^{\oplus a} \oplus \sco(-1)^{\oplus m-a} \right)$
	for $1 \leq a \leq m$, and so there are $m$ total such bundles. Altogether, there are $\sum_{m \mid n} m$ such bundles as $m$ ranges over
	the
	divisors of $n$.
	
\item A fourth heuristic for the average size of Selmer groups, 
	which comes from a heuristic distribution for the sizes of Selmer groups, is given in \cite{poonenR:random-maximal-isotropic-subspaces-and-selmer-groups}
	in terms of maximal isotropic subgroups of quadratic spaces.
	\end{enumerate}

	We next sketch why heuristics $(1)$, $(2)$, and $(3)$ above all yield the same average size $\sum_{m \mid n} m$.
	First we connect $(1)$ and $(3)$. The balanced bundles appearing in $(3)$ index the connected components of the moduli stack of $\pgl_m$ bundles
	on $\bp^1$. So, to identify $(1)$ and $(3)$ it suffices to show the number of
	connected components of the moduli stack of $\pgl_m$ bundles on $\bp^1$ equals $\tau(\pgl_m)$. This follows from \cite[Corollary 3.4]{behrendD:connected-components-moduli}.

	The equivalence between $(2)$ and $(3)$
	follows from piecing together \cite[Lemma 4.8, p. 785 line 20, Theorem 5.4, and Theorem 4.9]{de-jongF:on-the-geometry-of-principal-homogeneous-spaces}.
	To explain this briefly, 
	the quadratic form from $(2)$, which \cite{de-jongF:on-the-geometry-of-principal-homogeneous-spaces}
	notate as $\wp$, is induced by the Pontrjagin square map, and the value of this form on a Selmer element
	is related by
	\cite[Lemma 4.8 and p. 785 line 20]{de-jongF:on-the-geometry-of-principal-homogeneous-spaces} to the degree of a certain vector bundle.
	For general Selmer elements, the splitting type of this vector bundle is determined by its degree using that the bundle is balanced
	\cite[Theorem 5.4]{de-jongF:on-the-geometry-of-principal-homogeneous-spaces}.
	Conversely, every balanced splitting type of a vector bundle on $\bp^1$ corresponds to a component of the Selmer space by 
	\cite[Theorem 4.9]{de-jongF:on-the-geometry-of-principal-homogeneous-spaces}
	and so one can reverse the above process to determine the value of the quadratic form from the splitting type of the associated vector bundle.

	It would be interesting to better understand how heuristics (1), (2), and (3) are related to (4) above.
\end{remark}

\begin{remark}[Heuristics for distributions of Selmer groups]
	\label{remark:poonen-rains-heuristics}
	In addition to predictions for average sizes of Selmer groups,
	there are also predictions for the higher moments and distributions
	of Selmer groups, such as in \cite{poonenR:random-maximal-isotropic-subspaces-and-selmer-groups}
	and \cite{bhargavaKLPR:modeling-the-distribution-of-ranks-selmer-groups}.
	In collaboration with Tony Feng and Eric Rains \cite{fengLR:geometric-distribution-of-selmer-groups}, we built on
	\autoref{theorem:n-monodromy-image}
	to prove these predictions
	over function fields, in the large $q$ limit.

	While this ultimately boils down to a version of the Chebotarev density theorem,
	significant work beyond the results of this paper is needed.
	For example, in order determine the full distribution,
	we require a more precise version of \autoref{theorem:n-monodromy-image},
	which computes the relevant monodromy group exactly. 
	Note that in \autoref{theorem:n-monodromy-image}, we only show the relevant monodromy group contains
	a suitably large subgroup, which is sufficient for the purposes of computing the average size of $n$-Selmer groups,
	but not for determining the distribution of $n$-Selmer groups.
	It is also nontrivial to show the resulting distribution agrees with that predicted in
	\cite{bhargavaKLPR:modeling-the-distribution-of-ranks-selmer-groups}.
	We accomplish this by showing the two distributions agree modulo primes,
	and then prove they can be described in terms of the same Markov process to show they
	agree modulo prime powers.
\end{remark}
\begin{remark}[Relation to homological stability]
	\label{remark:homological-stability}
	We now explain how \autoref{theorem:number-components}
	can be viewed as a step toward determining
	$\average(\sel_n/\mathbb F_q(t))$.

	By the Lang-Weil estimate,
	the limit in \autoref{theorem:number-components} is really computing the
	number of irreducible components of a certain stack we call the ``$n$-Selmer stack of height $d$,''
	which we notate as $\sstack n d {\mathbb F_p}$.
	Essentially, the $\mathbb F_q$ points of $\sstack n d {\mathbb F_p}$
	parameterize $n$-Selmer
	elements on elliptic curves over $\mathbb F_q(t)$ of height $d$.
	In this sense, \autoref{theorem:number-components} is demonstrating cohomological stability
in $d$ for $H^0( \sstack n d {\ol{\mathbb F}_p}, \mathbb Q_\ell)$, which counts the number of irreducible components
	of $\sstack n d {\ol{\mathbb F}_p}$.

	If one could show the other cohomologies of $\sstack n d {\ol{\mathbb F}_p}$ 
	also stabilize in $d$, 
	this would go a long way toward showing
	$\average(\sel_n/\mathbb F_q(t))=\sum_{m \mid n} m$.
	Indeed, one could then try to apply the Grothendieck Lefschetz trace formula to compute
	the average size of Selmer groups,
	similarly to how 
	the average size of $\ell$-parts of class groups are computed in 
	\cite{EllenbergVW:cohenLenstra} (see also
	\cite{churchEF:representation-stability-finite-fields} and \cite{farbW:etale-homological-stability-and-arithmetic-statistics}).
	One may alternatively wish to demonstrate homological stability for the closely related stacks parameterizing points of
	$\mathcal A_{n,d}$ and $\mathcal B_{n,d}$, introduced in
	\cite[\S5.2 and \S5.3]{de-jong:counting-elliptic-surfaces-over-finite-fields}.
\end{remark}

\begin{remark}[Average sizes, with a twist!]
	\label{remark:}
	Let $\mathbb F_q$ be a finite field of characteristic more than $3$.
	Recall that if $E$ is an elliptic curve over $\mathbb F_q(t)$ 
	defined by $y^2z = x^3 + a_4(s,t)xz^2 + a_6(s,t)z^3$, one can define the quadratic twist family of degree $d$
	as those elliptic curves of the form
	$f(s,t)y^2z = x^3 + a_4(s,t)xz^2 + a_6(s,t)z^3$,
	for $f(s,t) \in k[s,t]$ varying over square-free homogeneous polynomial of degree $d$.
	This is a family over an open subscheme of affine space parameterized by the coefficients of $f(s,t)$.

	As we were working on this problem, we learned of forthcoming work of Park and Wang \cite{parkW:average-selmer-rank-in-quadratic-twist-families}. 
	They prove a
	result analogous to \autoref{theorem:number-components}, 
	at least in the case $n$ is a prime more than $3$,
	for
	particular quadratic twist families. 

	The method of \cite{parkW:average-selmer-rank-in-quadratic-twist-families} is similar to ours.
	The essential idea is to replace \autoref{theorem:n-monodromy-image}
	by the result of Hall \cite[Theorem 6.3]{hall:bigMonodromySympletic}. 
	In order to apply \cite[Theorem 6.3]{hall:bigMonodromySympletic},
	the authors work with a slightly different, but closely related to,
	the sheaf $\ssheaf n d B$ we consider in this paper.

	One can extend the result of \cite{parkW:average-selmer-rank-in-quadratic-twist-families} for prime $n$ more than $3$ to composite $n$ with $\gcd(n,6) = 1$, as we now sketch.
	For such $n$, one can calculate the monodromy of certain quadratic twist families
	by bootstrapping Hall's result for primes 
	to composites using \cite[Theorem 1.3]{vasiu2003surjectivity}
	and Goursat's lemma as in \cite[Proposition 2.5]{Greicius:surjective}.
	One can use this to deduce a variant of \autoref{theorem:number-components} for the quadratic twist families investigated in
	\cite{parkW:average-selmer-rank-in-quadratic-twist-families}, 
	which is applicable for composite $n$ (and not just prime $n$) with $\gcd(n,6) = 1$.

	Finally, 
	we mention one significant difference between the universal family examined in this paper and quadratic twist families.
	For the universal family, the monodromy computation from \cite[Theorem 4.10]{de-jongF:on-the-geometry-of-principal-homogeneous-spaces}
	was only carried out over $\bc$. Adapting this computation in characteristic $p \neq 2$ occupies the majority of
	\autoref{section:2-3-positive-characteristic}.
	In contrast, Hall's result \cite[Theorem 6.3]{hall:bigMonodromySympletic} already applies in positive characteristic.
	So, to compute the monodromy of quadratic twist families one does not have to grapple with the issue of relating the monodromy in characteristic $p$ to that
	in characteristic $0$. 
\end{remark}
\begin{remark}[$d=1$]
	\label{remark:}
	The assumption $d \geq 2$ in \autoref{theorem:number-components} is necessary.
	When $d = 1$, 
	the average size of $3$-Selmer groups in the large $q$ limit is $5$, as follows from
	\cite[Remark 7.8]{de-jong:counting-elliptic-surfaces-over-finite-fields}.
	It is an interesting open question to determine the average size for $n$-Selmer groups in the large $q$ limit when $d = 1$.
	When $n > 2$, the associated monodromy map in \autoref{definition:monodromy-representation}
	has image isomorphic to $W(E_8)$, the Weyl group of type $E_8$.
	In the case $n = 2$, the image is $W(E_8)/\left\{ \pm 1 \right\}$, as
	is explained in \cite[Proposition 4.2(a)]{vakil:twelve-points-on-the-projective-line}.
	The action of this group can be seen via its permutation of the $240$ lines
	on a del Pezzo $1$ surface, corresponding to the $240$ roots of $E_8$.
	The average size of these Selmer groups 
	can be interpreted geometrically in terms of possible configurations of $n$ lines on a del Pezzo $1$ surface 
	(see \cite[\S4.1, \S4.4, \S4.5, \S4.6, and \S5.1]{farb2018resolvent} for some tangentially related classical constructions)
	and also in terms of splitting types of rank $n$ projective bundles on $\bp^1$ (similarly to the $n = 3$ case
	carried out in \cite[\S7]{de-jong:counting-elliptic-surfaces-over-finite-fields}).
\end{remark}

\begin{remark}[Generalizations]
	\label{remark:}
	It would be interesting 
	to generalize \autoref{theorem:number-components} 
	to function fields of higher genus curves over finite fields.
	Many ideas for attempting this generalization may be found in \cite{ho-lehung-ngo:average-size-of-2-selmber-groups}.
	Perhaps the most obvious obstruction to this generalization is that our result depends on
	\cite[Theorem 4.9]{de-jongF:on-the-geometry-of-principal-homogeneous-spaces}, which is only stated for elliptic surfaces over $\bp^1_\bc$,
	and not over higher genus curves.
	
	Another possible research direction is be to prove a variant of \autoref{theorem:number-components} 
	for higher dimensional abelian varieties, instead of elliptic curves.
	As a related example, \cite{dao2017average} computes the average size of 
	$2$-Selmer groups of Jacobians of hyperelliptic curves over function fields. 
\end{remark}

\begin{remark}[$q \nmid 2$]
	\label{remark:char-2}
	We explain why we assume $q$ is relatively prime to $2$ in the statement of \autoref{theorem:number-components}.
	In order to prove \autoref{theorem:number-components}, the key is to compute the image of the monodromy representation of \autoref{definition:monodromy-representation}.
	To compare the monodromy representation over $\bc$ to the representation over a finite field, we verify tameness of the associated cover in the proof of
	\autoref{proposition:monodromy-tame}.
	One can verify that the cover is not tamely ramified in characteristic $2$, and so our proof does not immediately extend to characteristic $2$ fields.
\end{remark}

The general approach we take for proving \autoref{theorem:number-components} is to construct an appropriate covering space parameterizing Selmer elements.
We then determine the monodromy group of this covering space well enough to compute the number of components of this space and hence its number of points over a large finite field.
A similar strategy for the Cohen-Lenstra heuristics for torsion in class groups was 
originally taken up in unpublished notes of Jiu Kang Yu, further
developed by
Achter
\cite{Achter:cohenQuadratic,Achter:distributionClassGroups},
and built upon by
Ellenberg, Venkatesh, and Westerland 
\cite{EllenbergVW:cohenLenstra}.

\subsection{An outline of the proof of \autoref{theorem:number-components}}
\label{subsection:prime-idea}

For $k$ a finite field,
we construct an algebraic space $\sspace n d k$ parameterizing pairs $(E, X)$, where $E$ is an elliptic curve over $k(t)$ and $X$
is approximately (but not exactly) an $n$-Selmer element of $E$.
Letting $\espace d k$ denote a parameter space for Weierstrass equations of elliptic curves $E/k(t)$ of height $d$,
there is quasifinite \'etale map
$\sspace n d k \ra \espace d k$
sending $(E,X) \mapsto [E]$.
The key property of $\sspace n d k$ is that
for almost all elliptic curves $E$ over $k(t)$,
there is a bijection between $\sel_n(E)$ and
$k$ points of the fiber $[E]\times_{\espace d k} \sspace n d k$.

Therefore, computing the average size of $n$-Selmer groups in the large $q$ limit
is reduced to computing the ratio $\frac{\# \sspace n d k (k')}{\# \espace d k(k')}$
for sufficiently large finite extensions $k'$ of $k$.
By the Lang-Weil estimate, since $\espace d k$ is geometrically irreducible,
we can show this ratio is $\sum_{m \mid n} m$ by showing $\sspace n d k$ has $\sum_{m \mid n} m$ irreducible components, all of which are geometrically irreducible.

To compute the number of irreducible components of $\sspace n d k$, we show that, over a dense open
$\smespace d k \subset \espace d k$, $\sspace n d k$ is a finite \'etale cover. 
The geometric fibers of the resulting restriction $\smsspace n d k \ra \smespace d k$ are isomorphic to a free $\bz/n\bz$ module
$\vsel n d k$.
Hence, we obtain a monodromy representation (or Galois representation)
$\mono n d k: \pi_1^{\et}(\smespace d k) \ra \gl(\vsel n d k)$.
In fact, there is a certain quadratic form $\qsel n d k$ on $\vsel n d k$ coming from Poincar\'e duality
which is preserved under the monodromy action, and so the map $\mono n d k$ factors though
$\o(\qsel n d k)$.
In fact, $(\vsel n d k, \qsel n d k)$ is naturally the reduction $\bmod n$ of
a quadratic space $(\vsel \bz d \bc, \qsel \bz d \bc)$ over $\bz$.
Let $r_n: \o(\qsel \bz d \bc) \ra \o(\qsel n d \bc)$ denote the reduction $\bmod n$ map.
Over $k = \bc$, it is shown in \cite{de-jongF:on-the-geometry-of-principal-homogeneous-spaces} that
$\im \mono n d \bc$ is contained in $\o(\qsel n d \bc)$ and contains the subgroup $r_n(\osp(\qsel \bz d \bc)) \subset \o(\qsel n d \bc)$ given by the kernel
of the $-1$-spinor norm.
If we knew $\im \mono n d {\ol{\mathbb F}_p} = \im \mono n d \bc$,
the number of geometrically irreducible components
would simply be the number of orbits of this group on its underlying free $\bz/n\bz$ module, which is $\sum_{m \mid n} m$.

Hence, the only task remaining is to show the monodromy over $\ol{\mathbb F}_p$ agrees with that over $\bc$.
This will follow once we show that $\mono n d k$
factors through the tame fundamental group, meaning its ramification orders over boundary divisors in a compactification of $\smespace d k$ are prime to $\chr(k)$.
To set up this transfer of monodromy, we work relatively over $\spec \bz[1/2]$ 
and use a relative version of Abhyankar's lemma.
Abhyankar's lemma allows us to reduce to checking
that the relative boundary divisors over $\bz[1/2]$ are smooth over an open set meeting all fibers over $\spec \bz[1/2]$.
The most difficult boundary divisor to deal with is that parameterizing singular elliptic surfaces,
which can be written down explicitly in terms of equations. 
Hence, verifying the above mentioned smoothness properties boils
down to a concrete calculation using the Jacobian criterion for smoothness.

\subsection{Outline of the Paper}
\label{subsection:outline}

The rest of the paper is devoted to the proof of \autoref{theorem:number-components}.
Along the way, we introduce the $n$-Selmer space, notated $\sspace n d B$ (depending
on a height $d$ and a base scheme $B$), which may be of independent interest.
In \autoref{section:notation} we collect various notation used throughout the paper; 
\autoref{table:notation} may be useful.
In \autoref{section:selmer-space} we define the $n$-Selmer space,
develop its basic properties,
and explain the relation between points of the $n$-Selmer space and
$n$-Selmer groups of elliptic curves.
In \autoref{section:2-3-positive-characteristic} we compute the monodromy of the $n$-Selmer space
over the space of minimal Weierstrass models and
use this to show the $n$-Selmer space has $\sum_{m \mid n} m$ irreducible components, all of which are geometrically irreducible.
We combine our above computations
to prove \autoref{theorem:number-components} in \autoref{section:finishing-proof}.
See \autoref{figure:proof-schematic} for a schematic depiction of how the proof of \autoref{theorem:number-components} fits together.
\begin{figure}
	\centering
\begin{equation}
  \nonumber
  \begin{tikzcd}[row sep = tiny, column sep = 1.7em]
	  \qquad  & & & &  & 
	  \\
	  \qquad  & & \text{Cor.}~\ref{corollary:equality-selmer-fibers} \ar{ld}& \text{Prop.}~\ref{proposition:uniform-bound-selmer-fibers}\ar{l}
	  & \text{Lem.}~\ref{lemma:pic-and-pic0-inequality}\ar{l}
	  &    \\
	\qquad  & \text{Prop.}~\ref{proposition:stacky-selmer-count} \ar{ld}
	& \text{Cor.}~\ref{corollary:uniform-bound-selmer-fibers} \ar{l}  &  \text{Prop.}~\ref{proposition:selmer-space-and-h1} \ar{l}\ar{ul} & \text{Lem.}~\ref{lemma:selmer-ratio-equality} \ar{ul} & \\
	\text{Thm.}~\ref{theorem:number-components}  & & \text{Lem.}~\ref{lemma:algebraic-space-lang-weil} \ar{ul} & &   \text{Prop.}~\ref{lemma:selmer-space-is-almost-scheme}\ar{dl} \ar{ld} &  \text{Lem.}\ref{lemma:selmer-space-normalization} \ar{l} \\
	& \text{Cor.}~\ref{corollary:number-components} \ar{ul} &  \text{Def.}~\ref{definition:monodromy-representation} \ar{d}
	&\text{Cor.}~\ref{corollary:finite-etale-smooth-locus} \ar{l}   & \text{Lem.}~\ref{lemma:constructible-rank} \ar{l}
	&  \text{Lem.}~\ref{lemma:pi-2-smooth} \ar{d}\\
	\qquad  & & \text{Thm.}~\ref{theorem:n-monodromy-image} \ar{ul} & \text{Prop.}~\ref{proposition:monodromy-tame} \ar{l}{\text{\cite{de-jongF:on-the-geometry-of-principal-homogeneous-spaces}}}  \ar{l} & \text{Cor.}~\ref{lemma:2-divisor-complement} \ar{l}  &  \text{Prop.}~\ref{lemma:divii-geom-integral} \ar{l}\\
\end{tikzcd}
\end{equation}
\caption{
A schematic diagram depicting the structure of the proof of \autoref{theorem:number-components}.}
\label{figure:proof-schematic}
\end{figure}

\subsection{Acknowledgements}
I thank Ravi Vakil for guiding me down a fascinating path that eventually led to this
problem.
I thank Anand Deopurkar for pointing out the key reference \cite{de-jongF:on-the-geometry-of-principal-homogeneous-spaces}.
Thanks to 
Tony Feng and Eric Rains
for suggesting numerous improvements and for many illuminating conversations
involving extensions of the results of this paper.
I also thank an anonymous referee for many helpful comments.
Thanks to Arul Shankar for pointing out \autoref{corollary:average-rank}.
I thank
Brian Conrad,
Jordan Ellenberg, 
David Benjamin Lim, Anand Patel,
and
Bogdan Zavyalov
for some extremely valuable discussions.
I would also like to thank 
Manjul Bhargava,
K\k{e}stutis \v{C}esnavi\v{c}ius,
Aise Johan de Jong,
Fran\c{c}ois Greer,
Nikolas Kuhn, 
Bao L\^e H\`ung,
Sun Woo Park,
Eric Rains,
Arpon Raksit,
Zev Rosengarten,
Will Sawin,
Doug Ulmer,
Niudun Wang,
Melanie Wood,
and David Zywina
for helpful interactions.
This material is based upon work supported by the National Science Foundation Graduate Research Fellowship Program under Grant No. DGE-1656518.

\section{Notation}
\label{section:notation}

In this section, we collect various notation used throughout the paper.

\subsection{Notation for height}
\label{subsubsection:height}
We define a notion of height for elliptic curves over function fields, following
\cite[\S4.2-\S4.8]{de-jong:counting-elliptic-surfaces-over-finite-fields}.
Let $k$ be a field with $\chr(k) \neq 2$ and let $E$ be an elliptic curve over $k(t)$.
In this case, 
for 
$i \in \{1,2,3\}$
there is some $d \in \bz$ and
homogeneous polynomials $a_{2i}(s,t) \in k[s,t]$ of degree $2id$ 
so that
$E$ can be expressed in Weierstrass form as
\begin{align}
	\label{equation:weierstrass-form}
	y^2z = x^3 + a_2(s,t)x^2z + a_4(s,t)xz^2 + a_6(s,t)z^3.
\end{align}
By a change of coordinates, we can write $E$ in minimal Weierstrass form,
meaning there is no non-constant polynomial $f \in k[s,t]$ with
$f^{2i} \mid a_{2i}(s,t)$ for all $i \in \{1,2,3\}$.
Up to transformations of the form $x \mapsto x+r(s,t)$ for $r(s,t) \in k[s,t]$ of degree $2d$ and
$(a_2(s,t), a_4(s,t), a_6(s,t)) \mapsto (u^2a_2(s,t), u^4 a_4(s,t), u^6 a_6(s,t))$ for $u \in k^\times$,
elliptic curves over $k(t)$ have a unique such expression
in minimal Weierstrass form.
This follows from the standard procedure for simplifying Weierstrass equations, as described in
\cite[III.3.1]{Silverman2009}.
For $E$ written in minimal Weierstrass form,
the {\em discriminant} of $E$ is 
$\disc(E) := -16(4a_2(s,t)^3a_6(s,t) - a_2(s,t)^2 a_4(s,t)^2 + 4a_4(s,t)^3 + 27 a_6(s,t)^2 - 18a_2(s,t)a_4(s,t)a_6(s,t))$.
We define the {\em height} of $E$ as $h(E) := d = \deg \disc(E)/12$.
Note that $\disc(E)$ depends on the choice of Weierstrass form for $E$, but two different choices of minimal Weierstrass
form will yield two discriminants with the same degree, so $h(E)$ is an intrinsic invariant of $E$.

\subsection{Group theory notation}
\label{subsubsection:group-notation}

Let $V$ be a rank $s$ free module over a commutative ring $R$ with unit.
We let $\gl(V)$ or $\gl_{s}(R)$ denote the group of 
invertible $R$-homomorphisms $V \ra V$. 
If $V$ is a finite rank free module over $R$ with a quadratic form
$q$, let
$\o(q) \subset \gl(V)$ denote the associated orthogonal group. 

Suppose that $R = \bz$, $\varepsilon \in \{ \pm 1\}$, and
$(q,V)$ is a unimodular lattice, meaning $B_q$, viewed as a linear transformation $V \ra V^\vee$ is invertible.
Following \cite[\S5.1]{ebeling:monodromy-groups-isolateda},
define $\o^*_\varepsilon(q) \subset \o(q)$ to be the subset of those elements
$g \in \o(q)$ so that for any expression
$g = r_{v_1} \cdots r_{v_i}$ there is an even number of indices
$i$ with $\varepsilon q(v) <0$.
The group $\o^*_\varepsilon(q)$ is also known as the kernel of the $\varepsilon$-spinor norm.



\subsection{General notation throughout the paper}
\label{subsubsection:general-notation}
We collect some notation we shall use throughout the paper.
We will use $n$ as the integer indexing the Selmer group $\sel_n$, i.e., we work with the $n$-Selmer group.
For defining parameter spaces of elliptic curves we will work over a base scheme $B$ on which $2$ is invertible.
For defining parameter spaces of $n$-Selmer elements, we will further assume $n$ is invertible on $B$.
We often take $B$ to be $\spec k$ for $k$ a field, in which case we typically assume $\chr(k) \nmid 2n$ unless otherwise specified.
We use $d$ to denote the height of various elliptic surfaces, so that a minimal Weierstrass equation 
is of the form
$y^2 z = x^3 + a_2(s,t)x^2z + a_4(s,t)xz^2 + a_6(s,t)z^3$
for an elliptic curve
$E$
over $k(t)$ as in \eqref{equation:weierstrass-form}.
Here, for $i \in \{1,2,3\}$, $\deg a_{2i}(s,t) = 2id$ as 
homogeneous polynomials in $k[s,t]$.

For $X \ra Y$ and $Z \ra Y$ two maps, we notate $X_Z := X \times_Y Z$. When $Z = \spec R$ for $R$ a ring, we also notate
$X_R := X_{\spec R}$.
Similarly, while many objects throughout the paper are indexed by a base scheme $B$ (see those in \autoref{table:notation}),
if $B = \spec R$, we index them by $R$ instead. So, for example, we use $\espace d R$ (defined in \autoref{definition:weierstrass-family}) to mean $\espace d {\spec R}$,
and similarly for the other constructions in \autoref{table:notation}.

Throughout we use $H^i$ to denote \'etale cohomology, unless otherwise specified.
On certain occasions we will need both \'etale and group cohomology, which we will then notate
via subscripts $\et$ and $\grp,$.
Similarly, by $\pi_1$ we mean the \'etale fundamental group.
On occasion we will need the topological fundamental group, which we then notate as $\pi_1^{\top}$, and the tame
fundamental group, which we notate as $\pi_1^{\tame}$ (see \autoref{subsubsection:tame}). On these occasions,
we will notate the \'etale fundamental group as $\pi_1^{\et}$.

For $K$ a global field, we let $\places K$ denote the places of $K$.
We let $K_v$ denote the completion of $K$ at $v \in \places K$.
By {\em global function field}, we mean the fraction field of a smooth geometrically integral curve over $\mathbb F_q$.
Recall that when $K$ is a global function field, the elements of $\places K$ are in bijection with
the closed points of the smooth proper curve $C$ whose function field is $K$.
If $X$ is an integral ring or scheme, we let $K(X)$ denote its fraction field.
For $R$ a local ring, 
we let $R^\sh$ denote its strict henselization.

\subsection{Elliptic curves}
\label{subsubsection:elliptic-curve-notation}
	Let $B$ be a scheme with $2n$ invertible on $B$ and let $d \in \bz_{\geq 0}$.
	In \autoref{definition:weierstrass-family} and \autoref{definition:selmer-space} we define a scheme $\espace d B$ and an
	algebraic space $\sspace n d B$.
	These are parameter spaces for minimal Weierstrass models and elements of Selmer groups respectively.
	There is a natural map $\pi: \sspace n d B \ra \espace d B$.
Given a point (or geometric point) $x \in \espace d B$ 
	we let $E_x$ denote the elliptic curve corresponding to the point $x$,
	$\sce_x$ denote the N\'eron model of $E_x$ over $\bp^1_x$, $\sce_x^0$ denote its identity component,
	and $f_x : W_x \ra \bp^1_x$ denote the minimal Weierstrass model.
	See the proof of \autoref{lemma:selmer-space-stalks} for various relations between these objects.

	Let $g_x: \bp^1_x \ra x$ denote the structure map.
	For $x$ a point, we use $\ol{x}$ to denote a corresponding geometric point.
	So, if $x = \spec k,$ we let $\ol{x} = \spec \ol{k}$.
	Then, $E_{\ol x}$ denotes the corresponding elliptic curve over $\ol x$ and so on.

We use Kodaira's notation for types of singular fibers in minimal regular proper models of elliptic curves, see
\cite[IV.9, Table 4.1]{Silverman:Advanced}.

\subsection{Selmer groups}
\label{subsubsection:selmer-notation}

Let $E$ be an elliptic curve over a global function field $K$.
We let $\sel_n(E) := \ker\left(H^1(K, E[n]) \ra \prod_{v \in \places K} H^1(K_v, E)\right)$
denote the $n$-Selmer group of $E$.

\subsection{Tame fundamental group}
\label{subsubsection:tame}

We recall the definition of the tame fundamental group of a relative curve over a discrete valuation ring or field, following \cite[p. 9]{orgogozo2000theoreme} (see also \cite[Expos\'e XIII, 2.1.3]{noopsortSGA1Grothendieck1971}).
For $S$ a discrete valuation ring or field,
let $X \ra S$ be a regular relative curve, $E \subset X$ a divisor \'etale over $S$, and $V := X - E$.
Let $F \ra V$ be a finite \'etale cover and let $\ol{F}$ denote the normalization of $X$ along $F \ra V$.
We say $F \ra V$ is {\em tame} if the ramification orders $\ol{F}$ over $E$ are invertible on $S$.
We then define $\pi_1^{\tame}(V)$ as the profinite group whose finite quotients $G$ correspond to tame finite \'etale Galois-$G$ covers of $V$.
In particular, $\pi_1^{\tame}(V)$ is a quotient of $\pi_1^{\et}(V)$.

\subsection{Summary of notation introduced in the paper}
\label{subsubsection:notation-spaces}

For the reader's convenience, in \autoref{table:notation}
we collect notation introduced throughout the paper,
roughly in order of appearance.
\begin{figure}
	\centering
\scalebox{0.7}{
	\begin{tabular}{|l|l|l|}
		\hline
		Notation & Description & Location defined \\\hline
		$\aff d B$ & The affine space parameterizing the coefficients of $a_2, a_4$, and $a_6$ in the Weierstrass equation & \autoref{definition:weierstrass-family} \\ \hline
		$\espace d B$ & A parameter space of Weierstrass equations over $\bp^1$ of height $d$  & \autoref{definition:weierstrass-family} \\ \hline
		$\smespace d B$ & The open set of $\espace d B$ of smooth minimal Weierstrass models & \autoref{definition:smooth-espace} \\ \hline
		$\uespace d B$ & The universal family of Weierstrass models over $\espace d B$  & \autoref{definition:weierstrass-family} \\ \hline
		$\usmespace d B$ & The universal family of Weierstrass models over $\smespace d B$ & \autoref{definition:smooth-espace} \\ \hline
		$\sspace n d B$ & The $n$-Selmer space of height $d$ & \autoref{definition:selmer-space} \\ \hline
		$\smsspace n d B$ & The open subspace of $\sspace n d B$ given by restricting to $\smespace d B$ & \autoref{definition:selmer-space} \\ \hline
		$\ssheaf n d B$ & The sheaf on $\espace d B$ represented by $\sspace n d B$ & \autoref{lemma:representability-sspace} \\ \hline
		$\smssheaf n d B$ & The sheaf on $\smespace d B$ represented by $\smsspace n d B$ & \autoref{definition:smooth-espace} \\ \hline
		$\estack d B$ & The moduli stack of minimal Weierstrass models, a quotient of $\espace d B$ by an algebraic group & \autoref{definition:selmer-stack} \\ \hline
		$\smestack d B$ & The open substack of $\estack d B$ parameterizing smooth Weierstrass models & \autoref{definition:selmer-stack} \\ \hline
		$\sstack n d B$ & The $n$-Selmer stack of height $d$, a quotient of $\sspace n d B$ by an algebraic group & \autoref{definition:selmer-stack} \\ \hline
		$\smsstack n d B$ & The open substack of $\sstack n d B$ given by restricting to $\smestack d B$ & \autoref{definition:selmer-stack} \\ \hline
	$\comp E v$ & The group of rationally defined components of the N\'eron model of $E$ at a closed point $v$ & \autoref{definition:component-group} \\ \hline
		$\fcomp E$ & The product of $\comp E v$ over all closed points $v$ & \autoref{definition:component-group} \\ \hline
		$\mono n d B$ & The monodromy representation associated to $\smsspace n d B \ra \smespace d B$ & \autoref{definition:monodromy-representation} \\ \hline
		$\vsel n d B$ & The free $\bz/n\bz$ module corresponding to the geometric generic fiber of $\sspace n d B \ra \espace d B$ & \autoref{definition:vsel} \\ \hline
	$\qsel n d B$ & The quadratic form on $\vsel n d B$ respected by the geometric monodromy & \autoref{theorem:n-monodromy-image} \\ \hline
	$\divsing d B$ & The divisor in $\aff d B$ parameterizing singular elliptic surfaces & \autoref{definition:divii} \\ \hline
	\end{tabular}
}
	\vspace{.5cm}
	\caption{Notation introduced in the paper.}
	\label{table:notation}
\end{figure}

\section{The $n$-Selmer space}
\label{section:selmer-space}

We define the $n$-Selmer space and $n$-Selmer stack in \autoref{subsection:selmer-space-definition}
and
prove various properties of the $n$-Selmer space in \autoref{subsection:selmer-space-properties}.
In 
\autoref{subsection:selmer-space-points},
we relate points of the $n$-Selmer space to elements of Selmer groups of elliptic curves.

\subsection{Defining the $n$-Selmer space}
\label{subsection:selmer-space-definition}

Our next goal is to define the $n$-Selmer space and $n$-Selmer stack, which is accomplished in
\autoref{definition:selmer-space} and \autoref{definition:selmer-stack}.
We will realize the $n$-Selmer space as a cover of a parameter space for minimal Weierstrass equations of elliptic curves.
\subsubsection{Motivation for the definition of the $n$-Selmer space}
\label{subsubsection:motivation-selmer-space}
The motivation for our definition of the Selmer space is as follows.
Let $E$ be an elliptic curve over $k(t)$
with N\'eron model $\sce$.
Then, the $n$-Selmer group of $E$ is closely connected to $H^1(\bp^1, \sce^0[n])$ via results discussed in \autoref{subsection:selmer-space-points} below.
So, we will cook up a sheaf over the parameter space of height $d$ elliptic curves whose stalk over a point corresponding to $E$ is $H^1(\bp^1_{\ol k}, \sce_{\ol k}^0[n])$.

In order to define the $n$-Selmer space, 
we now define the relevant parameter space of Weierstrass equations of height $d$.

\begin{definition}
	\label{definition:weierstrass-family}
	Throughout this definition we work relatively over a base scheme $B$ on which $2$ is invertible.
	Define $\bp^1_B := \proj_B \sco_B[s,t]$.
	Form the affine space,
	\begin{align*}
		\aff d B := \spec_B \sco_B[a_{2,0}, a_{2,1} \ldots, a_{2,2d}, a_{4,0}, \ldots, a_{4,4d}, a_{6,0} \ldots, a_{6,6d}].
	\end{align*}
	For $i \in \{1,2,3\}$, define $a_{2i}(s,t) := \sum_{j=0}^{2id} a_{2i,j} t^j s^{2id-j}$.
	Let $\espace d B \subset \aff d B$ denote the open subscheme parameterizing those points so that the Weierstrass equation $y^2z = x^3 + a_2(s,t)x^2z + a_4(s,t)xz^2 + a_6(s,t)z^3$ is a minimal
	Weierstrass equation.
	This is open as it corresponds to those points $(a_{2,0}, \ldots, a_{2,2d}, a_{4,0}, \ldots, a_{4,4d}, a_{6,0} \ldots, a_{6,6d})$
	so that there is no point $p \in \bp^1_B$ with $\val_p\left(a_{2i}(s,t)\right) \geq 2id$ for all $i \in \{1,2,3\}$.
	
We next construct a family of minimal Weierstrass models over
$\espace d B$.
Consider the projective bundle $\proj_{\bp^1_B \times_B \espace d B} \sym^\bullet \sck$ with
	\begin{align}
		\label{equation:projective-bundle}
		\sck := \sco_{\bp^1_B \times_B \espace d B} \oplus \sco_{\bp^1_B \times_B \espace d B}(-2d) \oplus \sco_{\bp^1_B \times_B \espace d B}(-3d).
	\end{align}
	Let $z,x$, and $y$ denote the generators of the first, second, and third summands of $\sck$.
	Let $\uespace d B$
denote the subscheme cut out of $\proj_{\bp^1_B \times_B \espace d B} \sym^\bullet \sck$ by 
the ideal sheaf generated by the equation $y^2z = x^3 + a_2(s,t)x^2z + a_4(s,t)xz^2 + a_6(s,t)z^3$ in $\sym^3(\sck)$. Then, $\uespace d B$
is the family of height $d$ Weierstrass models over $\espace d B$.
\end{definition}
We can summarize the setup of \autoref{definition:weierstrass-family} by the diagram
\begin{equation*}
	\label{equation:espace-diagram}
	\begin{tikzcd} 
		\uespace d B \ar[hookrightarrow]{r} \ar{d}{f} & \proj_{\bp^1_B \times_B \espace d B} \sym^\bullet \sck \ar {dl}  \\
		\bp^1_B \times_B \espace d B \ar {d}{g} & \\
		\espace d B \ar[hookrightarrow] {r} & \aff d B.
\end{tikzcd}\end{equation*}

Suitably motivated to define the $n$-Selmer space over $\espace d B$ by \autoref{subsubsection:motivation-selmer-space},
we next construct a sheaf $\ssheaf n d B$ which is represented by the $n$-Selmer space.
\begin{lemma}
	\label{lemma:representability-sspace}
	Let $B$ be a scheme with $2n$ invertible on $B$, and $d \geq 0$.
	Let $\uespace d B \xra{f} \bp^1_B \times_B \espace d B \xra{g} \espace d B$ denote the projection maps, and define
	$\ssheaf n d B := R^1 g_* (R^1 f_* \mu_n)$.
	Then $\ssheaf n d B$ is a constructible sheaf of $\bz/n\bz$ modules
	whose formation commutes with arbitrary base change on $\espace d B$.
	Further, $\ssheaf n d B$ is represented by an algebraic space
	which is \'etale, quasi-separated, and of finite type over $\espace d B$.
\end{lemma}
\begin{proof}
This holds by proper base change together with standard facts about the algebraic space representing an \'etale sheaf
(see \cite[V, Theorem 1.5]{Milne:etaleBook}, \cite[p. 157, point (b)]{Milne:etaleBook} for quasi-separatedness, and \cite[V, Proposition 1.8]{Milne:etaleBook}).
\end{proof}

Using \autoref{lemma:representability-sspace}, we can now define the $n$-Selmer space.
\begin{definition}[The Selmer space]
	\label{definition:selmer-space}
	For $B$ a scheme with $2n$ invertible on $B$,
	define the {\em $n$-Selmer space over $B$ of height $d$},
	denoted $\sspace n d B$,
	to be the algebraic space representing the sheaf
$\ssheaf n d B$, as defined in \autoref{lemma:representability-sspace}.
\end{definition}

Points of the Selmer space do not quite correspond bijectively to Selmer elements. 
The main source of this discrepancy is due to the fact that isomorphism classes of elliptic curves
appear multiple times in $\espace d B$, as a given isomorphism class of elliptic curve has many Weierstrass equations.
In order to fix this discrepancy, we now introduce the Selmer stack.
Although it would be possible to prove our main result, \autoref{theorem:number-components}, only using the Selmer space,
we find it cleaner to introduce the Selmer stack, whose points correspond more closely to Selmer elements.
The Selmer space is a smooth cover of the Selmer stack of relative dimension $2d + 2$. 
For the remainder of the paper we will work almost exclusively with
the Selmer space, only interacting with the Selmer stack at the very end of the proof in \autoref{section:finishing-proof}.
The reader interested in understanding the proof of \autoref{theorem:number-components}
but unfamiliar with stacks can safely ignore the Selmer stack without detracting from their understanding of the proof.

\begin{definition}[The Selmer Stack]
	\label{definition:selmer-stack}
	Retain notation from \autoref{definition:weierstrass-family}.
	Let $d \geq 0$ and $B$ be a scheme with $2n$ invertible on $B$.
	We next construct a $\bg_a^{2d+1} \rtimes \bg_m$ action on $\uespace d B$.
	This action is given on Weierstrass equations as follows. Viewing an element $r \in \bg_a^{2d+1}$ as a homogeneous degree $2d$ polynomial in $s$ and $t$ whose coefficients are parameterized by 
	$\bg_a^{2d+1}$, $r$ sends
	\begin{align*}
		\resizebox{1.0\linewidth}{!}{
  \begin{minipage}{\linewidth}
	\begin{align*}
		y^2 z = x^3 + a_2(s,t)x^2z + a_4(s,t)xz^2 + a_6(s,t)z^3 \mapsto y^2 z = (x+r)^3 + a_2(s,t)(x + r)^2 z + a_4(s,t)(x + r)z^2 + a_6(s,t)z^3,
	\end{align*}
  \end{minipage}
}
\end{align*}
where one then simplifies the right hand side to determine the coefficients $a_{2i,j}$.
	The action of $\lambda \in \bg_m$ is given by 
\begin{align*}
		\resizebox{1.0\linewidth}{!}{
  \begin{minipage}{\linewidth}
	\begin{align*}
		y^2 z = x^3 + a_2(s,t)x^2z + a_4(s,t)xz^2 + a_6(s,t)z^3 \mapsto y^2 z = x^3 + \lambda^{2} a_2(s,t)x^2 z + \lambda^{4} a_4(s,t)xz^2 + \lambda^{6}a_6(s,t)z^3.
	\end{align*}
  \end{minipage}
}
\end{align*}
	
This $\bg_a^{2d +1} \rtimes \bg_m$ action on $\uespace d B$ induces actions on $\bp^1_B \times_B \espace d B$ and $\espace d B$ respecting the projection maps 
$\uespace d B \xra{f} \bp^1_B \times_B \espace d B \xra{g} \espace d B$.
It therefore induces an action on $\ssheaf n d B = R^1 f_* (R^1 g_* \mu_n)$ and hence an action on $\sspace n d B$.
We define the {\em moduli stack of height $d$ minimal Weierstrass models over $B$} as the quotient stack $\estack d B := \left[ \espace d B/ \bg_a^{2d+1} \rtimes \bg_m \right]$.
We define the {\em $n$-Selmer stack over $B$ of height $d$} as the quotient stack $\sstack n d B := \left[ \sspace n d B/ \bg_a^{2d+1} \rtimes \bg_m \right]$.
Since the action of $\bg_a^{2d+1} \rtimes \bg_m$ restricts to an action on $\usmespace d B$, (as defined later in \autoref{definition:smooth-espace},) we similarly define
$\smestack d B := \left[ \smespace d B/ \bg_a^{2d+1} \rtimes \bg_m \right]$
and
$\smsstack n d B := \left[ \smsspace n d B/ \bg_a^{2d+1} \rtimes \bg_m \right]$.
\end{definition}
\begin{remark}
	\label{remark:}
	The $\mathbb F_q$ points of the $n$-Selmer stack of height $d$ can be given a modular interpretation as certain pairs 
	$(f: Y \ra \bp^1, D)$
	where $f: Y \ra \bp^1$ is a proper flat map with smooth generic fiber and geometrically irreducible fibers,
	and $D \subset Y$ a Cartier divisor of relative degree $n$, satisfying additional conditions with various equivalences.
	We will not need this, so we do not precisely formulate the interpretation, but one can deduce it from the proof of
	\autoref{proposition:selmer-space-and-h1}	
	and \cite[Proposition 1.7]{artinSD:the-shafarevich-tate-conjecture-for-pencils-of-elliptic-curves} 
	(see also the somewhat more precise formulation in \cite[Lemma 4.2]{de-jongF:on-the-geometry-of-principal-homogeneous-spaces}). 
\end{remark}
\begin{remark}
	\label{remark:finite-stabilizers}
	The Selmer stack is always a smooth algebraic stack by 
	\cite[Example 8.1.12]{olsson2016algebraic},
	though it will fail to be Deligne-Mumford in characteristic $3$. 
	For example, in characteristic $3$,
	the $\bg_m$ action will have $\mu_3$ stabilizers on
	points corresponding to curves
	of the form $y^2z = x^3 + a_6(s,t)z^3$.
	Further, the automorphism group schemes of any point of the Selmer stack are always finite,
	as one can check via explicit computations in terms of Weierstrass equations.
	For example, the degree of the \'etale part of the stabilizer is computed in
	\cite[Proposition 1.2]{Silverman:AEC}
	and one can use similar computations to check the connected part is finite.
\end{remark}

\subsection{Properties of the Selmer space}
\label{subsection:selmer-space-properties}

In this subsection, we prove various properties of the Selmer space.
In \autoref{subsubsection:stalks} we describe the geometric fibers of $\sspace n d B$ over $\espace d B$.
In \autoref{subsubsection:opens}, we define an open subset $\smespace d B \subset \espace d B$,
fiberwise dense over $B$,
over which the Selmer space will be a finite \'etale cover of degree $n^{12d-4}$.
In \autoref{subsubsection:schematic}, we show the Selmer space is a separated scheme over $\smespace d k$.
In \autoref{subsubsection:finiteness} we show the Selmer space is finite over $\smespace d k$.

%

\subsubsection{Geometric fibers of the Selmer space}
\label{subsubsection:stalks}

We next describe the geometric fibers of $\sspace n d B \ra \espace d B$ in terms of various cohomology groups in \autoref{lemma:selmer-space-stalks}.
\begin{lemma}
	\label{lemma:selmer-space-stalks}
	Let $k$ be a field and $E$ an elliptic curve over $k(t)$.
	Let
	$h: X \ra \bp^1_k$ be the associated minimal regular proper model, $f: W \ra \bp^1_k$ be the associated minimal
	Weierstrass model, and $\sce$ be the associated N\'eron model over $\bp^1_k$.
	Let $W^{\sm}$ denote the smooth locus in $W$ of the map $f$.
	\begin{enumerate}
		\item We have isomorphisms 
			\begin{align*}
				R^1f_* \mu_n \simeq W^{\sm}[n] \simeq \pic_{W/\bp^1_k}^0[n]  \simeq \sce^0[n]\simeq \pic_{X/\bp^1_k}^0[n].
			\end{align*}
			If the total space of $W$ is smooth, the above are all isomorphic to $\sce[n]$.
		\item If $E$ corresponds to a point $x \in \espace d k$,
		letting $\ol{x}$ denote a geometric point over $x$, the geometric fiber of $\sspace n d B$ at $\ol x$ is isomorphic to 
	\begin{align*}
		H^1(\bp^1_{\ol x}, R^1 (f_{\ol x})_* \mu_n) &\simeq H^1(\bp^1_{\ol x}, W^{\sm}_{\ol x}[n]) 
		\simeq H^1(\bp^1_{\ol x}, \pic^0_{W_{\ol x}/\bp^1_{\ol x}}[n]) \\
		&\simeq H^1(\bp^1_{\ol x}, \sce_{\ol x}^0[n])
		\simeq H^1(\bp^1_{\ol x}, \pic_{X_{\ol x}/\bp^1_{\ol x}}^0[n]).
	\end{align*}
	If the total space of $W$ is smooth, the above are all isomorphic to $H^1(\bp^1_{\ol x} , \sce_{\ol x}[n])$.
	\end{enumerate}
\end{lemma}
\begin{proof}
	We first prove (1).
	The isomorphism $R^1f_* \mu_n \simeq W^{\sm}[n]$ is given in \cite[(1.3)]{artinSD:the-shafarevich-tate-conjecture-for-pencils-of-elliptic-curves}.
	Next, we claim that $\pic_{W/\bp^1_k}^0 \simeq W^{\sm} \simeq \sce^0\simeq \pic_{X/\bp^1_k}^0 ,$ which will imply the corresponding isomorphisms on $n$-torsion.
	The isomorphism $\pic_{W/\bp^1_k}^0 \simeq W^{\sm}$ is shown in \cite[(1.1) and (1.2)]{artinSD:the-shafarevich-tate-conjecture-for-pencils-of-elliptic-curves},
	$W^{\sm} \simeq \sce^0$ is shown in \cite[Corollary 9.3]{Silverman:Advanced},
	and $\sce^0 \simeq \pic_{X/\bp^1_k}^0$ is shown in \cite[\S9.5, Theorem 4(b)]{BoschLR:Neron}.
	In the case that the total space of $W$ is smooth, $W \simeq X$ since $W$ is regular with integral fibers, and we have $\pic^0_{W/\bp^1_k} \simeq \sce$
	by \cite[\S9.5, Theorem 1]{BoschLR:Neron}.
	Therefore, in this case, $\pic^0_{W/\bp^1_k}[n] \simeq \sce[n]$.

	We next prove (2).
	Let $g: \bp^1_x \ra x$ denote the structure morphism,
	let $g_{\ol x}: \bp^1_{\ol x} \ra \ol x$ its base change to $\ol x$, and let $f_{\ol x}$ denote the base change of $f$ to $\ol x$.
	By \autoref{lemma:representability-sspace} (or really just proper base change), the geometric fiber of $\sspace n d B$ over $\ol x$ is identified with the stalk of 
	$\ssheaf n d B$ at $\ol{x}$.
	By construction, this fiber is $R^1 (g_{\ol x})_*(R^1 (f_{\ol x})_* \mu_n) \simeq H^1(\bp^1_{\ol x}, R^1 (f_{\ol x})_* \mu_n)$.
	The various isomorphisms are consequences of the first part.
\end{proof}
\begin{remark}
	\label{remark:}
	In the setting of \autoref{lemma:selmer-space-stalks},
	one can also verify $\pic_{X/\bp^1}[n] = \pic^0_{X/\bp^1}[n]$.
	(In fact, this even holds more generally for minimal regular proper models
	of elliptic curves over Dedekind bases.)
	Hence, the stalk of the Selmer
	space at $x$ is also identified with $H^1(\bp^1_{\ol x}, \pic_{X_{\ol x}/\bp^1_{\ol x}}[n])$.
	This remark will not be needed in what follows.
\end{remark}

\subsubsection{An open subset of the Selmer space}
\label{subsubsection:opens}

In this section, we define the open set $\smespace d B \subset \espace d B$, which parameterizes smooth minimal Weierstrass models.
\begin{definition}
	\label{definition:smooth-espace}
	For $B$ a scheme with $2$ invertible,
	let $\smespace d B \subset \espace d B$ denote the open subscheme over which $\pi: \uespace d B \ra \espace d B$ is smooth.
	More formally, if $Z \subset \uespace d B$ denotes the singular locus of the map $\pi$, then let $\smespace d B := \espace d B - \pi(Z)$.
	Let $\usmespace d B := \uespace d B \times_{\espace d B} \smespace d B$.
	Let $\smsspace n d B := \sspace n d B \times_{\espace d B} \smespace d B$.
	Let $\smssheaf n d B := \ssheaf n d B|_{\smespace d B}$.
\end{definition}
\begin{remark}
	\label{remark:smooth-espace}
	Note that $\smespace d B \subset \espace d B$ is precisely the set of points where the Weierstrass model is already smooth,
	and hence isomorphic to the minimal regular proper model.
	By \cite[\S9.4, Theorem 4.35(a)]{liu:algebraic-geometry-and-arithmetic-curves}
	the minimal Weierstrass model is the minimal regular proper model if and only if all fibers of the minimal
	regular proper model are geometrically integral. 
	Hence, all fibers of a smooth Weierstrass model over $\bp^1$ either have good reduction,
	type $\mathrm{I_1}$ reduction, or type $\mathrm{II}$ reduction.
\end{remark}
Our next goal is to verify non-emptiness of $\smespace d B$, for any $B$ with $2$ invertible on $B$, which we do later in \autoref{lemma:reduced-discriminant}.
We now formally introduce the divisors which parameterize singular elliptic surfaces and elliptic surfaces with a cuspidal fiber.

\begin{definition}
	\label{definition:divii}
	Let $B$ be a scheme with $2$ invertible on $B$, let $d > 0$, and let $\aff d B$ denote the affine space of \autoref{definition:weierstrass-family} containing $\espace d B$.
Let $X_B := \proj_{\bp^1_B} \sym^\bullet \sco_{\bp^1_B} \oplus \sco_{\bp^1_B}(-2d) \oplus \sco_{\bp^1_B}(-3d)$.
Recall that every Weierstrass model $W$ of the form $y^2z = x^3 + a_2(s,t) x^2 z + a_4(s,t)xz^2 + a_6(s,t)z^3$ 
comes with an embedding $W \hookrightarrow X_B$.
Define
\begin{equation}
	\label{equation:incidence-divisor}
	\begin{tikzcd}
		\qquad &\Psi_B^d := \left\{ (W, p) : [W] \in \aff d B, p \in X_B, W\text{ is singular at }p\right\} \ar {ld}{\pi_1} \ar {rd}{\pi_2} & \\
		\aff d B && X_B. 
\end{tikzcd}\end{equation}
Define ${\mathcal C}_B^d$ as the subscheme of $\aff d B$ given as the scheme theoretic image of $\Psi^d_B$ under $\pi_1$.
Now, define $\divsing d {\bz[1/2]} := {\mathcal C}_{\bz[1/2]}^d$ and for a general $\bz[1/2]$ scheme $B$, define $\divsing d B := \divsing d {\bz[1/2]} \times_{\spec \bz[1/2]} B$.
\end{definition}
\begin{definition}
\label{definition:additive-divisor}
	Let $B$ be a scheme with $2$ invertible.
	Consider the incidence correspondence 
	\begin{equation}
	\label{equation:additive-incidence}
	\begin{tikzcd}[column sep = small]
	\qquad & \Phi_{B}^d := \{ ([W],p) : [W] \in \aff d B, p \in \bp^1_B, \text{$W \times_{\bp^1_B} p$ has a cusp} \} \ar {ld}{\sigma_1} \ar {rd}{\sigma_2} & \\
	\aff d {B} && \bp^1_{B}.
	\end{tikzcd}\end{equation}
	Define 
	$\divadd d {B}$ as the image of $\sigma_1$.
\end{definition}

Our next goal is to show that in each fiber over $\bz[1/2]$, a general
Weierstrass surface has squarefree discriminant.
The following lemma will be useful in computing the dimension of $\divsing d B$ over $B$,
and will also be used later in \autoref{subsection:boundary-divisors} where we need a more
detailed analysis of $\divsing d B$.

\begin{lemma}
\label{lemma:pi-2-smooth}
For $B$ a scheme with $2$ invertible on $B$ and $X_B$ the projective bundle over $B$ with coordinates $x,y,z$ as defined in
\eqref{equation:incidence-divisor},
let $Y_B \subset X_B$ denote the subscheme of $X_B$ given by $(y = 0)$ and $(z \neq 0)$.
Then, the map $\pi_2$ factors through the subscheme $Y_B \subset X_B$,
and the resulting map $\Psi_B^d \to Y_B$ is smooth of relative dimension $12d$ with geometrically integral fibers.
\end{lemma}
\begin{proof}
	Let $R$ be a ring and $p \in X_B(\spec R)$ be an $R$ valued point, which we may write in the form
	$p = [[x(p),y(p),z(p)],[s(p), t(p)]]$.
	In order to show $\pi_2$ factors through $Y_B$,
	we claim that if $W$ is singular at $p$ then $y(p) = 0$.
	To see this, if $z(p) = 0$, we then obtain $x(p) = 0$, which contradicts the fact that every point of the identity section of a Weierstrass model is smooth.
	Hence, we must have $z(p) \neq 0$, in which case we can view $p$ as a point on the open of the Weierstrass model 
	defined by $f := - y^2+  x^3 + a_2(s,t)x^2 + a_4(s,t) x + a_6(s,t)$.
	Since we must have $\frac{\partial f}{\partial y}(p) = 0$, we find $2y(p) = 0$ and so $y(p) = 0$ as $2$ is invertible on $B$.
	Ergo, the map $\pi_2$ factors through the locally closed subscheme $Y_B \subset X_B$ defined by $(y = 0)$ and $(z \neq 0)$.

	To conclude the proof, we claim $\Psi^d_B$ of \autoref{definition:divii}
	is smooth over $Y_B$ with fibers that are affine spaces of dimension $12d$.
	Let $R$ be a ring and $p \in Y_B(\spec R)$ and $R$-valued point.
	By changing coordinates on $\bp^1_{[s,t]}$, it suffices to check the above claim in a neighborhood of $\pi_2(p)$ for $p \in Y_B(\spec R)$ an $R$-valued point of $Y_B$ with $t(p) = 0$.
	Write $a_{2i}(s,t) = \sum_{j=0}^{2id} a_{2i,j} t^j s^{2id-j}$ for $i \in \{1,2,3\}$ as in \autoref{definition:weierstrass-family}.
	Using $t(p) = 0$, the condition that $W$ is singular at $p$ can be written in terms of the Jacobian criterion as those $(a_{2,0}, a_{2,1} \ldots, a_{2,2d}, a_{4,0}, \ldots, a_{4,4d}, a_{6,0} \ldots, a_{6,6d})$ satisfying 
	$f(p)=0$, $\frac{\partial f}{\partial x}(p)=0$, and $\frac{\partial f}{\partial t}(p)=0$
	(in addition to the condition that $y(p) = 0$ coming from $\frac{\partial f}{\partial y}(p)=0$).
Expanding these three equations yields
\begin{equation}
	\begin{aligned}
		\label{equation:jacobian-criterion}
		y^2 &= x^3 + a_{2,0} x^2 + a_{4,0} x + a_{6,0} \\
		0 &= 3x^2 + a_{2,0} x + a_{4,0} \\
		0 &= x^3 + a_{2,1} x^2 + a_{4,1} x + a_{6,1}
	\end{aligned}
\end{equation}
	at $p$.
	That is, when one replaces replace $x$ and $y$ in the above equations with $x(p)$ and $y(p)$, the vanishing locus
	of the resulting three equations defines those $W$ singular at $p$.
	These equations define three independent linear equations on the $a_{2i,j}$ coordinates of $\aff d B$.
	Therefore, $\pi_2$ is in fact an affine space bundle over $Y_B$ of relative dimension $12d$,
	and hence smooth of relative dimension $12d$ with geometrically integral fibers.
\end{proof}

We are now ready to show that a general Weierstrass surface has squarefree discriminant.
In particular, this will show that $\smespace d B$ is fiberwise dense in $\espace d B$.

\begin{lemma}
\label{lemma:reduced-discriminant}
	Let $k$ be a field of characteristic not equal to $2$.
	There is a nonempty open subscheme of $\espace d k$ parameterizing Weierstrass models with squarefree discriminant. 
	In particular, for any scheme $B$ with $2$ invertible on $B$, $\smespace d B$ is fiberwise dense over $B$, as is the open subscheme
	parameterizing smooth elliptic surfaces with no cuspidal fibers.
\end{lemma}
Note that the discriminant being squarefree is another way to say that the vanishing locus of the discriminant polynomial is reduced.
\begin{remark}
\label{remark:}
The natural analog of \autoref{lemma:reduced-discriminant} holds in characteristic $2$ as well.
Namely, it is true that a generic Weierstrass model has squarefree discriminant in characteristic $2$, and
essentially the same proof works. The only reason we did not state it above is that one has to first define
the relevant moduli space, which uses long Weierstrass form to describe the elliptic surfaces in characteristic $2$.
\end{remark}
\begin{proof}
	As a first step, we explain why a Weierstrass elliptic surface over a field $k$ of height $d$ has reduced discriminant if and only if it corresponds to a point in the complement 
	of $\divsing d k \cup \divadd d k$.
	Note that a smooth Weierstrass surface over $k$ which has no cuspidal fibers
	only has smooth fibers or $\mathrm{I_2}$ fibers, since being smooth implies the Weierstrass surface is
	the minimal regular proper model of its generic fiber.
	Such a surface has discriminant equal to the product of the places with $\mathrm{I_2}$ fibers, and therefore
	has squarefree discriminant.
	Conversely, because cuspidal fibers and singular points contribute a square factor to the discriminant,
	if the discriminant is squarefree, then the Weierstrass surface must be smooth 
	and without cuspidal fibers.

	Therefore, to conclude the proof, it is enough to show that the locus of Weierstrass surfaces of height $d$ over a field $k$
	which are smooth and have no cuspidal fiber is a dense open subscheme of $\espace d k$.
	The statement for general schemes $B$ then follows from base change to every point of $B$.

	Let $\aff d k$ denote the affine space of \autoref{definition:weierstrass-family}, and let $\divsing d k$ and $\divadd d k \subset \aff d k$ 
	be the subschemes defined in \autoref{definition:divii} and \autoref{definition:additive-divisor}.	
	By construction of these subschemes, the complement $\aff d k - (\divsing d k \cup \divadd d k)$
	is precisely the open subscheme whose points correspond to Weierstrass elliptic surfaces
	which are smooth and have no cuspidal fibers.
	Since $\dim \aff d k = 12d + 3$, to complete the proof, it is enough to show
	$\dim \divsing d k\cup \divadd d k \leq 12d + 2$. 
	(In fact, $\divsing d k$ and $\divadd d k$ are both irreducible of dimension $12d+2$, though we won't need to prove this.)

	First, we show $\divadd d k$ has dimension at most $12d+2$.
	Let $\sigma_1$ and $\sigma_2$ denote the maps from \autoref{definition:additive-divisor}.
	We claim the fiber of $\sigma_2$ over any point $p \in \bp^1_k$ has dimension $\dim \aff d k - 2$.
	To see this, the condition of having a cuspidal fiber over $p$
	can be expressed explicitly in terms of a Weierstrass equation $zy^2 = f(x,z)$
	as the condition that $f(x,z)$ has a single root in the fiber over $p$.
	The reduced scheme parameterizing degree $3$ polynomials with a single root has codimension $2$ in the scheme
	of all degree $3$ polynomials, which implies that the fibers of $\sigma_2$ have dimension $\dim \aff d k - 2$.
	Since $\divadd d k$ is the image of $\sigma_1$,
	we obtain $\divadd d k$ has dimension at most $\dim \aff d k - 2 + 1 = 12d+2$.

	To conclude, we verify
	$\divsing d k$ has dimension at most $12d+2$.
	This computation can be done similarly to the case of $\divadd d k$,
	though we opt to instead apply
	\autoref{lemma:pi-2-smooth}.
	The dimension of
	$\divsing d k$ is at most the dimension of
	$\Psi_k^d$ from
	\eqref{equation:incidence-divisor}.
	Further, defining $Y_k$ as in \autoref{lemma:pi-2-smooth},
	the dimension of $\Psi_k^d$ over $k$ is the sum of the dimension of $\Psi_k^d$ over $Y_k$ 
	and the dimension of $Y_k$.
	The former is $12d$ by \autoref{lemma:pi-2-smooth} and $\dim Y_k = 2$.
	Therefore, the dimension of $\Psi_k^d$, and hence $\divsing d k$ is at most $12d + 2$.
	\end{proof}
For the statement of the next lemma, recall our notational conventions from \autoref{subsubsection:general-notation}.
\begin{lemma}
	\label{lemma:torsion-points-proper}
	Let $d > 0$ and let $k$ be a field.
	For any map
	$\spec k \ra \smespace d B$, 
$H^0(\bp^1_{k}, \sce[n]) = H^0(\bp^1_{k}, \sce^0[n])= 0$.
\end{lemma}
\begin{proof}
	It suffices to check this statement at a geometric point $\ol x$ over $\spec k$.
	By \autoref{lemma:selmer-space-stalks}, for $\ol x \in \smespace d B$, we have 
$H^0(\bp^1_{\ol x}, \sce_{\ol x}[n]) \simeq H^0(\bp^1_{\ol x}, \sce^0_{\ol x}[n])$.
The claim then follows from \cite[Lemma 5.15]{de-jong:counting-elliptic-surfaces-over-finite-fields},
which says that when $d > 0$,
the only torsion section $\bp^1_{\ol x} \ra \sce_{\ol x}^0[n]$ is the identity section.
\end{proof}

\subsubsection{The schematic locus of the Selmer space}
\label{subsubsection:schematic}
We next show that the Selmer space is a separated scheme over $\smespace d B$, which we accomplish in
\autoref{lemma:selmer-space-is-almost-scheme}.
Separatedness will be crucially used in \autoref{subsubsection:finiteness} to show the Selmer space is finite over $\smespace d B$.
In order to verify separatedness, we will need the following lemma, which tells us that we can recover the relative $n$-torsion of a smooth elliptic surface
over a complete discrete valuation ring as the \'etale locus in its relative normalization over $\bp^1$.

\begin{lemma}
\label{lemma:selmer-space-normalization}
Let $B$ be a scheme with $2n$ invertible on $B$ and let $\usmespace d B \xra{f} \bp^1_B \times_B \smespace d B$
denote the projection map.
For a given map $T \to \smespace d B$ with $T$ the spectrum of a complete discrete valuation ring,
or $T$ a field of arbitrary characteristic,
let $f^T$ denote the base change of $f$ to $T$
and let $U_T$ denote the algebraic space representing the sheaf $R^1 f^T_* \mu_n$ on $\bp^1_T$.
Then, $U_T$ is represented by a scheme, quasi-finite and separated over $\bp^1_T$.
Further, letting $\nu_T : Z_T \to \bp^1_T$ denote the relative normalization
of $\bp^1_T$ along the map $U_T \to \bp^1_T$,
$U_T$ is the \'etale locus of $\nu_T: Z_T \to \bp^1_T$.
\end{lemma}
\begin{proof}
We first show $U_T$ is a scheme, separated and quasi-finite over $\bp^1_T$.
This follows from the fact that it has a locally closed embedding into $\usmespace d B \times_{\smespace d B} T$.
Indeed, this is explained in
\cite[p. 250-251]{artinSD:the-shafarevich-tate-conjecture-for-pencils-of-elliptic-curves},
the point being that the smooth locus of $\alpha: \usmespace d B \times_{\espace d B} T \to \bp^1_T$ can be identified with 
$\pic^0_{\alpha}$
and $U_T$ is identified with $\pic^0_{\alpha}[n]$.

Note that $U_T \to \bp^1_T$ is \'etale, and so $\nu_T$ is an isomorphism over $U_T$.
Hence, to show
$U_T$ is the \'etale locus of $\nu_T$,
it suffices
to show that $Z_T - U_T$ is contained in the singular locus of $\nu_T$.

Let $s$ denote the closed point of $T$ and $\eta$ denote the generic point.
Our ultimate goal will be to reduce to the case that $T = \spec k$ for $k$ a field,
and as preparation, we will check that $\nu_T$ is finite flat and that $U_s$ and $U_\eta$ are dense
in $Z_T \times_T s$ and $Z_T \times_T \eta$, respectively.
Since normalization does not in general commute with base change, it is not automatically true that $Z_T \times_T s \simeq Z_s$, and so this is not a tautology.

We now check $\nu_T$ is finite.
Since $T$ is the spectrum of a complete local noetherian ring, it is Nagata.
\cite[\href{https://stacks.math.columbia.edu/tag/032W}{Tag 032W}]{stacks-project}
Therefore, $\bp^1_T$ is also Nagata by \cite[\href{https://stacks.math.columbia.edu/tag/0334}{Tag 0334}]{stacks-project}.
Hence $\bp^1_T$ is universally Japanese again by \cite[\href{https://stacks.math.columbia.edu/tag/0334}{Tag 0334}]{stacks-project}.
By the definition of Japanese schemes, $\nu_T$ is finite.

We next check $\nu_T$ is flat.
Since $Z_T$ is a normal $2$-dimensional scheme it is Cohen-Macaulay.
Since $\bp^1_T$ is regular and $\nu_T$ is finite, it follows from
miracle flatness
\cite[\href{https://stacks.math.columbia.edu/tag/00R3}{Tag 00R3}]{stacks-project}
that $\nu_T$ is flat.

We next verify that $Z_T \times_T s$ and $Z_T \times_T \eta$ are proper curves birational to $U_s$ and $U_\eta$.
Properness is automatic from finiteness of $\nu_T$.
Further, since $U_T$ is dense in $Z_T$, we obtain that $U_\eta$ is dense in $Z_T \times_T \eta$,
and so these schemes are birational.
Next, we check $U_s$ is dense in $Z_T \times_T s$.
We show this by checking the 
degree of $U_s \to \bp^1_s$ agrees with the degree of $Z_T \times_T s \to \bp^1_s$.
Note that once we show the degrees are equal, it will follow that $U_s$ is dense in $Z_T \times_T s$
because both are pure dimension $1$ schemes due to $\nu_T$ being finite flat.
Since $\nu_T$ is a finite flat map of schemes each fiber of $\nu_T$ has the same degree.
Further, by construction of $U_T$ as $R^1 f_*^T \mu_n$, $n^2 = \deg(U_s \to \bp^1_s) = \deg(U_\eta \to \bp^1_\eta)$.
Because $\deg(U_\eta \to \bp^1_\eta) = \deg(Z_T \times_T \eta \to \bp^1_\eta) = \deg(Z_T \times_T s \to \bp^1_s)$,
the degree of $\nu_T$ is also equal to $n^2$.
Therefore, $\deg(U_s \to \bp^1_s) = n^2 =  \deg(Z_T \times_T s \to \bp^1_s)$ implying $U_s$ is dense in $Z_T \times_T s$.

We now use the above to reduce the case that $T$ is the spectrum of a complete 
discrete valuations ring to the case that $T = \spec k,$ for $k$ a field.
For any map $\spec k \ra T$, for $k$ a field,
we have seen above that $Z_T \times_T \spec k$ is a proper curve birational to $U_{\spec k}$.
Hence
$Z_{\spec k}$ is also the normalization of $Z_T \times_T \spec k$.
Therefore, we obtain a map $\phi_{\spec k \to T}: Z_{\spec k} \to Z_T \times_T \spec k$.
Observe that if $Z_T \times_T \spec k \to \bp^1_k$ were \'etale at a point $p$ of $Z_{\spec k} - U_{\spec k}$,
then $\phi_{\spec k \to T}$ is an isomorphism in a neighborhood of $p$.
Hence, it suffices to show $Z_{\spec k} - U_{\spec k}$ is contained in the singular locus of $\nu_{\spec k}$.
This completes the reduction to the case $T = \spec k$.

Now, suppose $p$ is a point of $Z_{\spec k} - U_{\spec k}.$ We will show $p$ lies in the singular locus of $\nu_{\spec k}$.
Let $W := \usmespace d B \times_{\smespace d B} \spec k$ denote the Weierstrass model of the elliptic curve associated to the map $\spec k \to \smespace d B$.
Then, the locally closed embedding $U_{\spec k} \to W$ given in 
\autoref{lemma:selmer-space-stalks}(1)
factors through a closed embedding into the N\'eron model of the generic fiber of $W \to \bp^1$, which is $W^{\sm}$, the smooth locus of $W \to \bp^1_k$, by
\autoref{lemma:selmer-space-stalks}(1).
If $Z_{\spec k} \to \bp^1_k$ were \'etale at $p$, the N\'eron mapping property would guarantee that $p$ factored through $W^{\sm}$, the N\'eron
model. 
However,
the map $U_{\spec k} \to W$ extends to a map $Z_{\spec k} \to W$ by normality of $Z_{\spec k}$ and properness of $W$.
Yet, $U_{\spec k} \subset W^{\sm}$ is closed in the N\'eron model.
This implies $p$ maps to a point of $W - W^{\sm}$,
and hence $Z_{\spec k} \to \bp^1_k$ cannot be \'etale at $p$.
\end{proof}

With the above lemma in hand, we are ready to prove $\smsspace n d B$ is separated.
The idea of the proof is to check separatedness by the valuative criterion. 
We check the valuative criterion by translating it to checking that a certain map of cohomology groups is injective.
We verify this in turn using \autoref{lemma:selmer-space-normalization} to recover a torsor for the $n$-torsion of an elliptic curve
as the \'etale locus in its relative normalization over $\bp^1$.

\begin{proposition}
	\label{lemma:selmer-space-is-almost-scheme}
	Let $B$ be a scheme with $2n$ invertible on $B$ and $d > 0$. Then $\smsspace n d B\ra \smespace d B$ is separated.
	In particular, $\smsspace n d B $, which is
	a priori only an algebraic space, is a scheme.
\end{proposition}
\begin{proof}
	By \cite[V, Theorem 1.5]{Milne:etaleBook} in order to show $\smsspace n d B$ is a scheme,
	it suffices to show 
	$\smsspace n d B\ra \smespace d B$ is separated.
	Further, since this map is the base change along $B \to \spec \bz[1/2n]$, it is enough to show
	$\psi: \smsspace n d {\bz[1/2n]} \ra \smespace d {\bz[1/2n]}$ is separated.
	By a variant of the valuative criterion for separatedness applied to only complete discrete valuation rings
	\cite[\href{https://stacks.math.columbia.edu/tag/0ARL}{Tag 0ARL}]{stacks-project}, 
	it is enough to check that for 
	all complete discrete valuation rings $R$ with $S = \spec R$, generic point $\eta$, and closed point $s$, it is enough to show
	the natural map $H^0(S, \smsspace n d {\bz[1/2n]}) \ra H^0(\eta, \smsspace n d {\bz[1/2n]})$ is injective.

For any scheme $T$ with a map $T \ra \espace d {\bz[1/2n]}$, we let $f^T$ and $g^T$ denote the base changes of the
morphisms $f$ and $g$ defined in the statement of \autoref{lemma:representability-sspace} to $T$.
We next claim that for any integral scheme $T$ with a map $T \ra \smespace d {\bz[1/2n]}$, the Leray spectral sequence associated to the composition of functors $\Gamma$ and $g^T_*$ 
induces an isomorphism $H^1(\bp^1_T, R^1 f^T_* \mu_n) \simeq H^0(T, R^1 g^T_* (R^1 f^T_* \mu_n))$.
To verify this, we only need verify that $H^1(T, g^T_* (R^1f^T_* \mu_n)) = H^2(T, g^T_* (R^1f^T_* \mu_n)) = 0$.
In fact, we will show $g^T_* (R^1 f^T_* \mu_n) = 0$.
Since $f^T$ and $g^T$ are proper, by proper base change, we can compute the stalks of the above sheaf after base change to any field valued point
of $T$, and so we may reduce to the case that $T= \spec k$ for $k$ a field.
In this case, we have $g^T_* R^1 f^T_* \mu_n = 0$ by \autoref{lemma:torsion-points-proper}.

The established claim of the previous paragraph applied to the cases $T = S$ and $T = \eta$ (for $S$ the spectrum of a complete discrete valuation ring)
shows that in order to check
$H^0(S, R^1 g^S_* (R^1 f^S_* \mu_n)) \ra H^0(\eta, R^1 g^\eta_* (R^1 f^\eta_* \mu_n))$ is injective, it is equivalent to check
\begin{equation}
H^1(\bp^1_S, R^1 f^S_* \mu_n) \ra H^1(\bp^1_\eta, R^1 f^\eta_* \mu_n)
\label{equation:generic-fiber-restriction}
\end{equation}
is injective.
Let $X$ denote the algebraic space corresponding to an element of $H^1(\bp^1_S, R^1 f^S_* \mu_n)$, thought of as an $R^1 f^S_* \mu_n$ torsor over $\bp^1_S$.
Let 
$\widetilde{X}$ denote the relative normalization of $\bp^1_S$ in $X \times_S \eta$.
We will show that $X$ is the \'etale locus of $\widetilde{X} \to \bp^1_S$.
This will imply
\eqref{equation:generic-fiber-restriction}
is injective as $X$ can be recovered from its image under
\eqref{equation:generic-fiber-restriction}.

As a first step, we check $X$ is open in $\widetilde{X}$.
Indeed, we have that $R^1 f^S_* \mu_n$ is representable by a quasi-finite separated scheme over $\bp^1_S$
by \autoref{lemma:selmer-space-normalization}.
Therefore, $X$ is also representable by a quasi-finite separated scheme over $\bp^1_S$, 
as this can be checked on an \'etale cover of $\bp^1_S$ over which $X$ acquires
a section.
By a variant of Zariski's main theorem
\cite[\href{https://stacks.math.columbia.edu/tag/082J}{Tag 082J}]{stacks-project}
it follows that $X$ is open in $\widetilde{X}$.

Using that $X$ is open in $\widetilde{X}$, we now complete the proof by verifying $X$ is precisely the \'etale locus of
$\widetilde{X} \to \bp^1_S$. 
Indeed, we can check this on an \'etale cover of $\bp^1_S$,
since the formation of normalization commutes with \'etale base change.
Hence, we may check this on a cover of $\bp^1_S$ over which $X$ becomes the trivial torsor, 
in which case the claim follows from \autoref{lemma:selmer-space-normalization}.
\end{proof}

\subsubsection{Finiteness properties of the Selmer space}
\label{subsubsection:finiteness}

We next aim to prove \autoref{corollary:finite-etale-smooth-locus}, 
which states that $\smsspace n d B \ra \smespace d B$ is finite and represents a sheaf of free $\bz/n\bz$ modules.
Our strategy for doing so is to prove this statement for every point of $\espace d B$,
and use this to deduce the finiteness statement by flatness considerations and constancy of degree.
In order to prove that it is a sheaf of {\em free} $\bz/n\bz$ modules, it will be useful to introduce the component group of the N\'eron model.
This component group will also play a crucial role in \autoref{subsection:selmer-space-points} when relating elements of Selmer groups to points of the Selmer space.

\begin{definition}
	\label{definition:component-group}
	Let $C$ be an integral Dedekind scheme and let $E$ be an elliptic curve over $K(C)$ with N\'eron model $\sce$.
	Let $v \in C$ be a closed point with residue field $\kappa(v)$.
	Let $\sce_v/\sce_v^0$ denote the group scheme of connected components of the special fiber.
	Define the group $\comp E v := (\sce_v/\sce_v^0)(\kappa(v))$
	and let $\fcomp E := \prod_{\text{closed points }v \in C} \comp E v$.
\end{definition}
\begin{remark}
	\label{remark:}
	The number
	$\# \comp E v$ is the local Tamagawa number of $E$ at $v$ and $\# \fcomp E$ is the Tamagawa number of $E$.
\end{remark}

Using \autoref{definition:component-group}, we can determine the module structure
of the geometric fibers of $\sspace n d k \ra \espace d k$ over a point of $\smespace d k$.
\begin{lemma}
	\label{lemma:constructible-rank}
	Let $k$ be an algebraically closed field of characteristic prime to $n$ (possibly of characteristic $2$)
	and let $E$ be an elliptic curve over $k(t)$ with N\'eron model $\sce \ra \bp^1_k$ and discriminant of degree $12d$ with $d >0$.
	\begin{enumerate}
		\item If $\fcomp {E}/n \fcomp {E}= \id$ and $H^0(\bp^1_k, \sce[n]) = 0$, then $H^1(\bp^1_k, \sce[n])$ is a free $\bz/n\bz$ module.
		\item If the minimal Weierstrass model of $E$ is smooth, $H^1(\bp^1_k, \sce[n])$ has rank $12d-4$.
	\end{enumerate}
\end{lemma}
\begin{proof}
	If $n = \prod_{i=1}^m p_i^{j_i}$, for $p_i$ distinct primes, we can write $\sce[n] \simeq \oplus_{i=1}^m \sce[p_i^{j_i}]$ as a product of sheaves.
	Since cohomology commutes with direct sums, we can therefore reduce both parts to the case
	that $n = p^j$ for $p$ a prime.

	We now prove $(1)$, assuming $n = p^j$.
		Let $r$ denote the rank of $H^1(\bp^1_k, \sce[p])$. That is, say $H^1(\bp^1_k, \sce[p]) \simeq \left( \bz/p\bz \right)^r$.
	We will show $H^1(\bp^1_k, \sce[p^j])$ is a free $\bz/p^j\bz$ module of rank $r$.
	For $0 \leq t \leq j$, we claim there is an exact sequence of sheaves
	\begin{equation}
		\label{equation:exact-torsion}
		\begin{tikzcd}
			0 \ar {r} &   \sce[p^t] \ar {r} & \sce[p^j] \ar {r} & \sce[p^{j-t}] \ar {r} & 0.
		\end{tikzcd}\end{equation}
	This is left exact as it is a pushforward of an analogous sequence associated to $E$.
	To see it is right exact, it suffices to show $\times p^t : \sce \ra \sce$ is surjective.
	But the cokernel of this map is identified with $\fcomp {E}/p^t\fcomp {E}$ 
	(see \cite[p. 629, line 12]{colliot-theleneSSD:hasse-principle-for-pencils-of-curves}), which we are assuming is trivial.

	By assumption, for all $t \leq j$, 
	$H^0(\bp^1_k, \sce[p^t]) = 0$. 
	We claim that also
	$H^2(\bp^1_k, \sce[p^t]) = 0$.
	Indeed, for $\iota: \spec k(t) \ra \bp^1_k$ the inclusion of the generic point, we have an isomorphism
	$\iota_* E[p^t] \simeq \sce[p^t]$ on the small \'etale site of $\bp^1_k$, as follows from the N\'eron mapping property for $\sce$.
	Hence, we want to show $H^2(\bp^1_k, \iota_* E[p^t]) = 0$.
	By Poincar\'e duality 
	\cite[V Proposition 2.2(b)]{Milne:etaleBook} 
	we only need check $H^0(\bp^1_k, \iota_* ((E[p^t])^\vee (1)) ) = 0$.
	The Weil pairing on the smooth elliptic curve $E$ gives an isomorphism $(E[p^t])^\vee (1) \simeq E[p^t]$.
	Therefore, we indeed have
	\begin{align*}
		H^2(\bp^1_k, \iota_* E[p^t]) \simeq H^0(\bp^1_k, \iota_* ((E[p^t])^\vee (1)) )^\vee \simeq H^0(\bp^1_k, \iota_* E[p^t])^\vee \simeq H^0(\bp^1_k, \sce[p^t])^\vee = 0.
	\end{align*}

	Therefore, \eqref{equation:exact-torsion} induces an exact sequence on cohomology
	\begin{equation}
		\label{equation:reverse-mu-i}
		\begin{tikzcd}
			0 \ar {r} & H^1(\bp^1_k, \sce[p^t])  \ar {r}{\mu^t} &  H^1(\bp^1_k, \sce[p^j]) \ar {r}{\upsilon^t} & H^1(\bp^1_k, \sce[p^{j-t}])  \ar {r} & 0.
		\end{tikzcd}\end{equation}
	By induction on $t$, we see 
	$H^1(\bp^1_k, \sce[p^j])$ is a $\bz/p^j\bz$ module of size $p^{jr}$.
	We next show it is free of rank $r$.

	Since finite $\bz/p^j\bz$ modules are all sums of $\bz/p^t\bz$ for $t \leq j$, to conclude the proof, we only need to show
	that the kernel of the multiplication by $p^{j-1}$ map $H^1(\bp^1_k, \sce[p^j]) \xra{\times p^{j-1}} H^1(\bp^1_k, \sce[p^j])$
	has size $p^{(j-1)r}$.
	To this end, note that the multiplication map by $p^{j-1}$ map on coefficients $\phi : \sce[p^j] \ra \sce[p^j]$
	induces the multiplication map $\times p^{j-1}$ on cohomology. 
	Observe that $H^1(\phi)$ factors as $H^1(\bp^1_k, \sce[p^j]) \xra{\upsilon^{j-1}} H^1(\bp^1_k, \sce[p]) \xra{\mu^1} H^1(\bp^1_k,\sce[p^j]),$ with $\mu^t, \upsilon^t$ defined in
	\eqref{equation:reverse-mu-i}.
	Taking $t= 1$ in \eqref{equation:reverse-mu-i} shows $\mu^1$ is injective. Therefore, $\ker(\times p^{j-1}) =\ker (\upsilon^{j-1} \circ \mu^1) = \ker \upsilon^{j-1}$. 
	In turn, applying \eqref{equation:reverse-mu-i} again with $t= j-1$, we see $H^1(\bp^1_k, \sce[p^{j-1}]) =\ker \upsilon^{j-1}$.
	By induction, this has order $p^{(j-1)r}$.
	Therefore, multiplication by $p^{j-1}$ on $H^1(\bp^1_k, \sce[p^j])$ has kernel of order 
	$p^{(j-1)r}$, implying $H^1(\bp^1_k, \sce[p^j])$ must be a free $\bz/p^j\bz$ module of rank $r$. This finishes $(1)$.
	
	We now prove $(2)$.
	Since $E$ has smooth minimal Weierstrass model, all fibers of the N\'eron model are integral.
	Since we are assuming $d > 0$,
	$H^0(\bp^1_k, \sce[n]) = 0$
	by \cite[Lemma 5.15]{de-jong:counting-elliptic-surfaces-over-finite-fields}.
	Hence, part $(1)$ applies.
	So, it suffices to show $(2)$ in the case $n$ is prime.
	Let $f_v(\sce[n])$ denote the {\em exponent of the conductor} of $\sce[n]$ at $v$.
	By	
	\cite[Th\'eor\`eme 1]{raynaud:characteristic-deuler-poincare} (where $f_v$ is notated as $\varepsilon_v^R$)
	the Euler characteristic of $\sce[n]$ is
	$-\sum_{v \in \places {K(\bp^1_k)}} f_v(\sce[n]) + 2 \cdot 2$.
	Since $H^i(\bp^1_k, \sce[n]) = 0$ for $i \neq 1$, as we showed in the proof of (1),
	the Euler characteristic of $\sce[n]$ is negative the
	rank of
	$H^1(\bp^1_k, \sce[n]).$
	By \cite[Theorem 2]{ogg:elliptic-curves-and-wild-ramification}, since all fibers of the N\'eron model are integral as $E$ has smooth minimal Weierstrass model,
	the exponent of the conductor is equal to the degree of the discriminant.
	Hence,
		\[
			\rk_{\bz/n\bz} H^1(\bp^1_k, \sce[n])= \sum_{v \in \places {K(\bp^1_k)}} f_v(\sce[n]) - 2 \cdot 2= \deg \disc(E) -4 = 12d-4.\qedhere
\]
\end{proof}

The following corollary is essential for defining our monodromy representation later in \autoref{definition:monodromy-representation} to count
the number of irreducible components of the Selmer space.
\begin{corollary}
	\label{corollary:finite-etale-smooth-locus}
	Suppose $B$ is a noetherian scheme with $2n$ invertible and $d > 0$.
	Then, $\pi: \smsspace n d B \ra \smespace d B$
	is finite \'etale, representing a locally constant constructible sheaf of rank $12d-4$ free $\bz/n\bz$ modules.
\end{corollary}
\begin{proof}
	We first show $\pi$ is finite.
	By 
\cite[II, Lemma 1.19]{deligneR:les-schemas-de-modules-de-courbes-elliptiques},
a quasi-finite flat separated morphism over a noetherian base scheme with constant fiber rank is finite.
Recalling our notational conventions from \autoref{subsubsection:elliptic-curve-notation},
the fiber of $\pi$ over a geometric point $\ol x \in \smespace d B$ is identified with $H^1(\bp^1_{\ol x}, \sce_{\ol x}[n])$ by \autoref{lemma:selmer-space-stalks}.
It follows that $\pi$
has constant fiber rank $n^{12d-4}$ by \autoref{lemma:constructible-rank}.
Further, $\pi$ is separated by \autoref{lemma:selmer-space-is-almost-scheme}.
Therefore, $\pi$ is finite.

Also, $\pi$ is \'etale by \autoref{lemma:representability-sspace}.
Since $\pi$ is finite \'etale, it represents a locally constant constructible sheaf of
rank $12d-4$ free $\bz/n\bz$ modules by \autoref{lemma:constructible-rank}.
\end{proof}

\subsection{Points of the Selmer space}
\label{subsection:selmer-space-points}

The main point of introducing the $n$-Selmer space is that points of the Selmer space parameterize elements of Selmer groups of the corresponding elliptic curves,
as we show in this subsection.

The two main results of this section are \autoref{corollary:equality-selmer-fibers} and \autoref{corollary:uniform-bound-selmer-fibers}.
These are the only two results of this section we will need for the proof of
\autoref{theorem:number-components}.
The first shows that the size of the Selmer group agrees with the number of $\mathbb F_q$ points of the fibers of $\smsspace n d B \ra \smespace d B$.
The second gives a uniform bound for the size of the Selmer group in terms of the number of $\mathbb F_q$ points of the fibers of $\sspace n d B \ra \espace d B$.
To prove these results, we will relate the number of $\mathbb F_q$ points of the fiber over $x \in \espace d B(\mathbb F_q)$ to $H^1(\bp^1_x, \sce_x[n])$
in \autoref{proposition:selmer-space-and-h1}.

\begin{proposition}
	\label{proposition:selmer-space-and-h1}
	With notation as in \autoref{subsubsection:elliptic-curve-notation}, suppose $d > 0$ and $x \in \espace d B(\mathbb F_q)$. Then, 
\begin{align}
		\label{equation:selmer-as-product-points}
		\# H^1(\bp^1_x, \sce^0_x[n]) = 
		\# \left(\pi^{-1}(x)\left( \mathbb F_q \right) \right).
	\end{align}
	If $d = 0$, we have $\# H^1(\bp^1_x, \sce^0_x[n]) = \# H^0(\bp^1_x, \sce^0_x[n])$.
	\end{proposition}
\begin{proof}
	To start, we give a cohomological description of 
	$\# \pi^{-1}(x)\left( \mathbb F_q \right)$.
	By \autoref{lemma:selmer-space-stalks},
	the geometric fiber of $\sspace n d B$ over $\ol x$ is $H^1(\bp^1_{\ol x}, \sce^0_{\ol x}[n])$.
To distinguish between \'etale and group cohomology, we use
	$H^i_{\grp}$ denote group cohomology and $H^i_{\et}$ to denote \'etale cohomology.
	Let $G_x := \aut(\bp^1_{\ol{x}}/\bp^1_{x})$. 
	The $\mathbb F_q$ points of $\pi^{-1}(x)$ are the $G_x$ invariants of $H^1_{\et}(\mathbb P^1_{\ol{x}}, \sce^0_{\ol{x}}[n])$. 
	That is, 
	$\pi^{-1}(x) (\mathbb F_q) = H_{\grp}^0(G_x, H^1_{\et}( \bp^1_{\ol {x}}, \sce^0_{\ol{x}}[n]))$.
	
	We relate this group to $H^1(\bp^1_x, \sce^0_x[n])$ using the Leray spectral sequence
	\begin{equation}
		\label{equation:selmer-to-points}
\begin{tikzpicture}[baseline= (a).base]
\node[scale=.84] (a) at (0,0){
	\begin{tikzcd}[column sep = tiny]
		0 \ar {r} & H^1_{\grp}( G_x, H^0_{\et}(\bp^1_{\ol x}, \sce^0_{\ol{x}}[n])) \ar {r} & H^1_{\et}( \bp^1_x, \sce^0[n]) \ar {r}{\theta} & H^0_{\grp}( G_x, H^1_{\et}( \bp^1_{\ol x}, \sce^0_{\ol{x}}[n]))  \ar {r} & H^2_{\grp}( G_x, H^0_{\et}(\bp^1_{\ol x}, \sce^0_{\ol{x}}[n])).
	\end{tikzcd}};
\end{tikzpicture}
\end{equation}
When $d > 0$,
we want to show $\theta$ is an isomorphism, so it suffices to 
show
$H^0_{\et}(\bp^1_{\ol x}, \sce^0_{\ol{x}}[n]) = 0$.
This holds by \cite[Lemma 5.15]{de-jong:counting-elliptic-surfaces-over-finite-fields},
applicable as $d > 0$.

Finally, the statement for $d = 0$ holds because $\sce^0_{\ol x}[n] \ra \bp^1_{\ol x}$ is finite so
$H^1_{\et}(\bp^1_{\ol x}, \sce^0_{\ol x}[n]) = 0$ and
\begin{align*}
	\# H^1_{\grp}(G_{x}, H^0_{\et}(\bp^1_{\ol x}, \sce^0_{\ol x}[n])) &= \# H^0_{\grp}(G_x, H^0_{\et}(\bp^1_{\ol x}, \sce^0_{\ol x}[n])) = \# H^0_{\et}(\bp^1_x, \sce_x^0[n]).\qedhere
\end{align*}
\end{proof}
Using \autoref{proposition:selmer-space-and-h1} we obtain the following precise 
relation between $\mathbb F_q$ points of the fiber of $\sspace n d B\ra \espace d B$
over $x$
when $x \in \smespace d B$.

\begin{corollary}
	\label{corollary:equality-selmer-fibers}
	With notation as in \autoref{subsubsection:elliptic-curve-notation}, suppose that $d > 0$.
	If $x \in \smespace d B(\mathbb F_q)$,
	\begin{align}
		\label{equation:selmer-equality-on-open}
		\# \sel_n(E_x) = \# \left(\pi^{-1}(x)\left( \mathbb F_q \right) \right).
	\end{align}
\end{corollary}
\begin{proof}
	Recall $\# \fcomp {E_x} = 1$ for $x \in \smespace d B(\mathbb F_q)$ by \autoref{remark:smooth-espace}. It follows from 
	\cite[Proposition 5.4(c)]{cesnavicius:selmer-groups-as-flat-cohomology-groups} that
	$\sel_n(E_x) \simeq H^1(\bp^1_x, \sce_x[n])$.
Here we are using the identification between  \'etale and fppf cohomology from \cite[Th\'eor\`eme 11.7 $1^{\circ}$]{grothendieck:brauer-iii}.
By \autoref{lemma:selmer-space-stalks}(1), $H^1(\bp^1_x, \sce_x[n]) \simeq H^1(\bp^1_x, \sce_x^0[n])$ for $x \in \smespace d B$,
and so the claim then follows from 
	\autoref{proposition:selmer-space-and-h1}.
\end{proof}

\subsubsection{An upper bound for sizes of Selmer groups}
\label{subsubsection:upper-bound}
Our only remaining goal in this section is to prove \autoref{corollary:uniform-bound-selmer-fibers}
which gives a uniform bound on the size of the Selmer group in terms of 
the number of $\mathbb F_q$ points of the fiber of $\sspace n d B\ra \espace d B$ over $x$.
This will be useful for counting Selmer elements associated to $x \notin \smespace d B$.
This will follow fairly immediately from \autoref{proposition:uniform-bound-selmer-fibers} which
bounds the size of the Selmer group in terms of the size of
$H^1(C, \sce^0[n])$ (which we relate to the Selmer space via \autoref{proposition:selmer-space-and-h1}).

To state \autoref{proposition:uniform-bound-selmer-fibers} we introduce
some notation.

\begin{notation}
	\label{notation:selmer-comparison}
	Fix $n \geq 1$ and a finite field $\mathbb F_q$ with $(q,n) = 1$.
	Let $C$ be a smooth proper geometrically connected curve over $\mathbb F_q$.
	Let $E$ be an elliptic curve over the function field $K(C)$.
	Let $\sce$ denote the N\'eron model of $E$ over $C$ with identity component
	$\sce^0$.
Let $X$ denote the minimal regular proper model of $E$ over $C$.
\end{notation}

In this paper, we will only use the following \autoref{proposition:uniform-bound-selmer-fibers}
over $C = \bp^1_{\mathbb F_q}$. However, we state it for general
curves $C$ over $\mathbb F_q$ as the proof is no harder.
We note that there appears to be a gap in the proof of the closely related \cite[Proposition 4.3.5]{ho-lehung-ngo:average-size-of-2-selmber-groups}
at the point where the proof of \cite[Proposition 4.3.4]{ho-lehung-ngo:average-size-of-2-selmber-groups} is referenced.
We would like to thank Bao L\^e H\`ung for confirming that the $n = 2$ case of
\autoref{proposition:uniform-bound-selmer-fibers} provides an alternate proof of
\cite[Proposition 4.3.5]{ho-lehung-ngo:average-size-of-2-selmber-groups}.

\begin{proposition}
	\label{proposition:uniform-bound-selmer-fibers}
With notation as in \autoref{notation:selmer-comparison},
	\begin{align}
		\label{equation:selmer-uniform-bound}
		\# \sel_n(E) \leq 
		\# H^0\left( C, \sce[n] \right) \cdot
		\# H^1(C, \sce^0[n]).
	\end{align}
\end{proposition}
\begin{proof}[Proof assuming \autoref{lemma:selmer-ratio-equality} and \autoref{lemma:pic-and-pic0-inequality}]
	By plugging in \autoref{lemma:pic-and-pic0-inequality} into \autoref{lemma:selmer-ratio-equality},
	and using the trivial inequality $H^0(C, \sce^0[n]) \geq 1$, we conclude $\frac{\# \sel_n(E)}{\# H^1(C, \sce^0[n])} \leq \# H^0(C, \sce[n])$ and hence
	\eqref{equation:selmer-uniform-bound} holds.
\end{proof}
See \autoref{example:selmer-inequality-achieved} for an example where 
equality in \autoref{proposition:uniform-bound-selmer-fibers} is achieved with $H^0(C, \sce[n]) \neq 0$.
Before proving \autoref{lemma:selmer-ratio-equality} and \autoref{lemma:pic-and-pic0-inequality}, we deduce a corollary we will need to relate points of the Selmer space to 
elements of Selmer groups.
\begin{corollary}
	\label{corollary:uniform-bound-selmer-fibers}
	Let $d > 0$.
	With notation as in \autoref{subsubsection:elliptic-curve-notation}, suppose $x \in \espace d B(\mathbb F_q)$. Then,
	\begin{align}
		\label{equation:selmer-numerical-bound}
		\# \sel_n(E_x) \leq 
		\# H^0\left( \bp^1_x, \sce_x[n] \right) \cdot
		\#\left(\pi^{-1}(x)\left( \mathbb F_q \right)\right)
		\leq n^2\cdot
		\#\left(\pi^{-1}(x)\left( \mathbb F_q \right)\right).
	\end{align}
\end{corollary}
\begin{proof}
	The first inequality follows by plugging in
	the result of
	\autoref{proposition:selmer-space-and-h1}
	to \autoref{proposition:uniform-bound-selmer-fibers}.
	The second inequality holds as 
	$\# H^0\left( \bp^1_x, \sce_x[n] \right) = \# E_x[n](\mathbb F_q(t)) \leq n^2$.
\end{proof}

To finish the proof of 
\autoref{proposition:uniform-bound-selmer-fibers}, we now prove
\autoref{lemma:selmer-ratio-equality} and \autoref{lemma:pic-and-pic0-inequality}.

\begin{lemma}
	\label{lemma:selmer-ratio-equality}
	With notation as in \autoref{notation:selmer-comparison},
	\begin{align}
		\label{equation:selmer-ratio-equality}
		\frac{\#\sel_n(E)}{\#H^1(C, \sce^0[n])} = \frac{ \# H^0(C, \sce[n])}{ \# H^0(C, \sce^0[n])}\cdot \frac{\# \Sha(E)[n]}{\# H^1(C, \sce^0)[n]}.
		\end{align}
\end{lemma}
\begin{proof}
	The proof will be a sequence of diagram chases.
To start, observe that both 
$\#\sel_n(E)$ and $\#H^1(C, \sce^0[n])$ are finite, the former by 
\cite[Theorem 8.4.6]{poonen:rational-points-on-varieties}
and the latter by \autoref{proposition:selmer-space-and-h1} (being the number of $\mathbb F_q$ points of the fiber of a quasi-finite map).

We claim there are exact sequences
\begin{equation}
		\label{equation:selmer-to-sha}
	\begin{tikzcd}[column sep = small]
		0 \ar {r} & \frac{H^0(C, \sce)}{nH^0(C, \sce)} \ar {r} & \sel_n(E)  \ar {r} & \Sha(E)[n] \ar {r} & 0 \\
\end{tikzcd}\end{equation}
\vspace{-1cm}
\begin{equation}
		\label{equation:sspace-to-pic}
	\begin{tikzcd}[column sep = small]
		0 \ar {r} & \frac{H^0(C, \sce^0)}{nH^0(C, \sce^0)} \ar {r} & H^1(C, \sce^0[n]) \ar {r} & H^1(C, \sce^0)[n] \ar {r} & 0.
\end{tikzcd}\end{equation}
Indeed, \eqref{equation:selmer-to-sha} follows from \cite[Theorem 4.2(a)]{Silverman:AEC} 
(whose proof is exactly the same over global fields as over number fields)
while \eqref{equation:sspace-to-pic} comes from taking cohomology associated to the multiplication by $n$ sequence on $\sce^0$.

By applying both \eqref{equation:selmer-to-sha} and \eqref{equation:sspace-to-pic}, we find
\begin{align}
		\label{equation:selmer-ratio-to-npic-and-sha}
		\frac{\#\sel_n(E)}{\#H^1(C, \sce^0[n])} = \frac{ \# \frac{H^0(C, \sce)}{nH^0(C, \sce)}}{ \# \frac{H^0(C, \sce^0)}{n H^0(C, \sce^0)}}\cdot \frac{\#\Sha(E)[n]}{\# H^1(C, \sce^0)[n]}.
\end{align}
To conclude, it suffices to check 
\begin{align}
	\label{equation:neron-ratio}
\frac{ \# \frac{H^0(C, \sce)}{nH^0(C, \sce)}}{ \# \frac{H^0(C, \sce^0)}{n H^0(C, \sce^0)}} = 
	\frac{ \# H^0(C, \sce[n])}{ \# H^0(C, \sce^0[n])}
\end{align}
For this, we now set up the commutative diagram with vertical maps $\alpha, \beta,$ and $\gamma$ defined below, each given by multiplication by $n$:
\begin{equation}
	\label{equation:neron0-and-neron}
	\begin{tikzcd}
		0 \ar {r} & H^0(C, \sce^0) \ar {r} \ar {d}{\alpha} & H^0(C, \sce) \ar {r} \ar {d}{\beta} &  H^0(C, \sce)/H^0(C, \sce^0) \ar {r} \ar {d}{\gamma} & 0 \\
		0 \ar {r} & H^0(C, \sce^0) \ar {r} & H^0(C, \sce) \ar {r} & H^0(C, \sce)/H^0(C, \sce^0) \ar {r} & 0.
\end{tikzcd}\end{equation}
Note that $H^0(C, \sce)/H^0(C, \sce^0)$ is a finite group, being a subgroup of the finite group $H^0(C, \sce/\sce^0)$. Hence $\# \ker \gamma= \#\coker \gamma$.
Also, the kernels and cokernels of $\alpha$ and $\beta$ are finite as $H^0(C, \sce) = H^0(K(C), E)$ is a finitely generated abelian group.
From the snake lemma we have $\# \ker \alpha \cdot \# \ker \gamma\cdot \#\coker \beta= \# \ker \beta \cdot \# \coker \alpha\cdot \#\coker \gamma$
and hence
$ \# \ker \alpha\cdot \#\coker \beta=\# \ker \beta \cdot \# \coker \alpha$.
Rearranging yields $\frac{\# \coker \beta }{\# \coker \alpha} = \frac{\# \ker \beta}{\# \ker \alpha}$.
This is precisely \eqref{equation:neron-ratio} because $\ker \beta = H^0(C, \sce[n])$ 
and $\ker \alpha= H^0(C, \sce^0[n])$.
\end{proof}

We now prove \autoref{lemma:pic-and-pic0-inequality}, whose proof
completes the proof of \autoref{proposition:uniform-bound-selmer-fibers}.

\begin{lemma}
	\label{lemma:pic-and-pic0-inequality}
With notation as in \autoref{notation:selmer-comparison},
	\begin{align}
		\label{equation:pic-and-pic0-inequality}
	\# \Sha(E)[n] \leq \# H^1(C, \sce^0)[n],
	\end{align}
	with equality holding if $\frac{H^0(C, \sce/\sce^0)}{H^0(C, \sce)}[n] = 0$.
\end{lemma}
\begin{proof}
	Using \cite[Proposition 4.5(b), (c), and (d)]{cesnavicius:selmer-groups-as-flat-cohomology-groups}, we have the equality 
	$\Sha(E) = \im \left( H^1(C, \sce^0) \ra H^1(C, \sce) \right).$
	Therefore, by taking cohomology associated to 
	\begin{equation}
		\label{equation:}
		\begin{tikzcd}
			0 \ar {r} & \sce^0 \ar {r} & \sce \ar {r} & \sce/\sce^0 \ar {r} & 0,
	\end{tikzcd}\end{equation}
	we obtain an exact sequence
	\begin{equation}
		\label{equation:sha-cohomology-sequence}
		\begin{tikzcd}
			0 \ar {r} & \frac{H^0(C, \sce/\sce^0)}{H^0(C, \sce)} \ar {r} & H^1(C, \sce^0) \ar {r} & \Sha(E) \ar {r} & 0.
	\end{tikzcd}\end{equation}
	To simplify notation, let $K := \frac{H^0(C, \sce/\sce^0)}{H^0(C, \sce)}$.
	Note that $K$ is finite because $\sce/\sce^0$ is a finite group scheme.
	Sending \eqref{equation:sha-cohomology-sequence} to itself via multiplication by $n$ and applying the snake lemma, we obtain an exact sequence
	\begin{equation}
		\label{equation:k-snake-lemma}
		\begin{tikzcd}
			0 \ar {r} & K[n] \ar {r} & H^1(C, \sce^0)[n] \ar {r} & \Sha(E)[n] \ar {r} & K/nK.
	\end{tikzcd}\end{equation}
	Therefore, $\# K[n] \cdot \# \Sha(E)[n] \leq \# H^1(C, \sce^0)[n] \cdot \# K/nK$.
	However, since $K$ is a finite group, $\# K[n] = \# K/nK$, and hence
	$\# \Sha(E)[n] \leq \# H^1(C, \sce^0)[n]$.
	Since $\# K/nK = \#K[n]$, \eqref{equation:k-snake-lemma} yields equality in \eqref{equation:pic-and-pic0-inequality} if $K[n] = 0$.
\end{proof}

To conclude, we give an example where equality in \autoref{proposition:uniform-bound-selmer-fibers} is achieved, but $H^0(C, \sce[n]) \neq 0$.
\begin{example}
	\label{example:selmer-inequality-achieved}
	With notation as in \autoref{notation:selmer-comparison},
	it is indeed possible that $\sel_n(E) > H^1(C, \sce^0[n])$ in the statement of \autoref{proposition:uniform-bound-selmer-fibers}
	when $E$ has nontrivial torsion points defined over $K(C)$.
	We produce an example in the case $C = \bp^1_{\mathbb F_7}$.
	Observe that by \cite[Theorem 5.1]{coxP:torsion-in-elliptic-curves},
	in order for $H^0(\bp^1_k, \sce[n]) \neq 0$, $n$ must have a prime factor $\leq 7$.
	We claim 
	$\sel_3(E) > H^1(C, \sce^0[3])$
	for $E$ the elliptic curve over $\mathbb F_7(t)$ defined by $y^2z + txyz + (t^3 + 3)yz^2 = x^3$.
	The \texttt{Magma} \cite{Magma} code
\begin{verbatimtab}
	F<t> :=FunctionField(GF(7));
	E := EllipticCurve([t,0,t^3+3,0,0]); 
	MordellWeilGroup(E);
	LocalInformation(E);
\end{verbatimtab}
	verifies $H^0(\bp^1_{\mathbb F_7}, \sce) = E(\mathbb F_7(t)) = \bz/3\bz$, $\sce/\sce^0 \simeq \bz/3\bz$, and further $E$ has two places
	of bad reduction: one of type $\mathrm{I}_1$ and the other of type $\mathrm{I}_3$.
	By \cite[Lemma 5.15]{de-jong:counting-elliptic-surfaces-over-finite-fields}, we have $H^0(\bp^1_{\mathbb F_7}, \sce^0[3]) = 0$
	and so the nontrivial $3$ torsion points necessarily meet the two non-identity
	components of $\sce$ in the fiber of type $\mathrm{I}_3$ reduction. Hence,
	the map $H^0(\bp^1_{\mathbb F_7}, \sce) \ra H^0(\bp^1_{\mathbb F_7}, \sce/\sce^0)$
	is surjective. Therefore, the inequality of \autoref{lemma:pic-and-pic0-inequality} is an equality.
	Combining this with $\#H^0(\bp^1_{\mathbb F_7}, \sce^0) = 1, \#H^0(\bp^1_{\mathbb F_7}, \sce) =3,$ and \autoref{lemma:selmer-ratio-equality}, we find
	$\# \sel_3(E) = \# H^0(\bp^1_{\mathbb F_7}, \sce[3]) \cdot \# H^1(C, \sce^0[3]) = 3 \cdot \# H^1(C, \sce^0[3]).$
\end{example}

\section{The monodromy of the Selmer space}
\label{section:2-3-positive-characteristic}

In this section, for $k$ a field with $\chr(k) \nmid 2n$, we constrain
$\mgeom n d k$, as defined in \autoref{definition:monodromy-representation}, in \autoref{theorem:n-monodromy-image}. 
The resulting corollary, \autoref{corollary:number-components}, 
is the only result of this section which will be used to prove \autoref{theorem:number-components} in \autoref{section:finishing-proof}.

The outline of this section is as follows.
In \autoref{subsection:setup-form} we define the monodromy representation associated to the Selmer space
and state the main result of this section, \autoref{theorem:n-monodromy-image}.
In \autoref{subsection:boundary-divisors} we analyze the divisor $\divsing d {\bz[1/2]}$ in a compactification of $\smespace d {\bz[1/2]}$
and prove this divisor is smooth over an open set meeting all fibers over $\spec \bz[1/2]$.
In \autoref{subsection:proof-monodromy-tame}
we compare the monodromy in characteristic $0$ and characteristic $p$.
Finally, in \autoref{subsection:monodromy-image} we prove \autoref{theorem:n-monodromy-image} and use this
to compute the number of components of the $n$-Selmer space.

\subsection{The monodromy representation}
\label{subsection:setup-form}

In \autoref{definition:selmer-space}, we constructed an algebraic space $\sspace n d B$
which is quasifinite over the parameter space for Weierstrass equations $\espace d B$.
By \autoref{corollary:finite-etale-smooth-locus}, we know that $\sspace n d B$ represents a locally constant sheaf 
of rank 
$12d-4$ free $\bz/n\bz$ modules over the open $\smespace d B \subset \espace d B$ defined in \autoref{definition:smooth-espace}.
We now introduce notation for the free module on which $\pi_1(\smespace d B)$ acts and then define the associated monodromy representation.

\begin{definition}
	\label{definition:vsel}
	Let $E$ be an elliptic curve over $k(t)$, for $k$ a field with $\chr(k) \nmid n$. 
Let $j: \spec k(t) \hookrightarrow \bp^1_k$ denote the 
inclusion of the generic point.
Let $\vs E n := H^1(\bp^1_{\ol k}, j_* E_{\ol k}[n])$.
If $B$ is an integral base scheme with $2n$ invertible on $B$, $\eta$
denotes the generic point of $\bp^1_B \times_B \espace d B$, and $f: \uespace d B \ra \bp^1_B \times_B \espace d k$ is the natural map, define
$\vsel n d B := \vs {f^{-1}(\eta)} n.$
\end{definition}

Using this, we can define the monodromy representation associated to the Selmer space.

\begin{definition}
	\label{definition:monodromy-representation}
	For $d > 0$,
	and $B$ an integral noetherian scheme,
	the locally constant rank $12d-4$ sheaf of free $\bz/n\bz$ modules $\smssheaf n d B$
	induces the monodromy representation (or Galois representation)
\begin{align}
	\label{equation:selmer-galois-rep}
	\mono n d B : \pi_1(\smespace d B)
	\ra \gl(\vsel n d k) \simeq \gl_{12d-4}(\bz/n\bz).
\end{align}
For an integral noetherian scheme $B$ with geometric generic point $\ol \eta$, we call $\im \mono n d B \subset \gl_{12d-4}(\bz/n\bz)$ the {\em monodromy} of the $n$-Selmer space of height $d$ over $B$
and $\im \mono n d {\ol \eta} \subset \gl_{12d-4}(\bz/n\bz)$ the {\em geometric monodromy} of the $n$-Selmer space of height $d$ over $B$.
\end{definition}

\begin{remark}
	\label{remark:}
	Technically speaking, we should keep track of base points in our
	fundamental groups.
	However, as we will ultimately be concerned with integral base schemes $B$,
	changing basepoint only changes the map $\mono n d k$ by conjugation.
	Since
	we will only care about the image of $\mono n d k$,
	we will omit the basepoint from our notation.
\end{remark}

Recall from \autoref{subsubsection:group-notation} that for $(V, Q)$ a quadratic space over $\bz$, we have
$\osp(Q) \subset \o(Q)$ defined
as the kernel of the $-1$-spinor norm.
The main result of this section is the following, which is proven in \autoref{subsection:monodromy-image}.

\begin{theorem}
	\label{theorem:n-monodromy-image}
	Suppose that $k$ is a field of characteristic prime to $2n$.
	For $d \geq 2$, 
	there is a non-degenerate quadratic space $(\vsel \bz d k, \qsel \bz d k)$ over
	$\bz$ whose reduction $\bmod n$ is $(\vsel n d k, \qsel n d k) :=
	(\vsel \bz d k \otimes_\bz \bz/n\bz, \qsel \bz d k \otimes_\bz \bz/n\bz)$,
such that the following holds.
	Let $r_n : \o(\qsel \bz d k) \ra \o(\qsel n d k)$ denote the induced 
reduction 
	$\bmod n$
	map.
	Then, 
	the images of the monodromy representation
	$\mono n d {k}: \pi_1(\smespace d {k}) \ra \gl(\vsel n d k)$
	and geometric monodromy representation 
	$\mono n d {\ol k}: \pi_1(\smespace d {\ol k}) \ra \gl(\vsel n d k)$
	of \autoref{definition:monodromy-representation}
	satisfy $r_n(\osp(\qsel \bz d \bz)) \subset \im \mono n d {\ol k} \subset  \im \mono n d {k} \subset \o(\qsel n d {\ol k})$.
\end{theorem}
\begin{remark}
	\label{remark:quadratic-form-determination}
	The quadratic form $\qsel n d k$ appearing in the statement of \autoref{theorem:n-monodromy-image}
	is explicitly determined in
	\cite[p. 786]{de-jongF:on-the-geometry-of-principal-homogeneous-spaces}.
	That is, $\qsel n d k$ is the reduction $\bmod n$ of the quadratic form $\qsel {\bz} d k$ associated to $U^{\oplus (2d-2)} \oplus (-E_8)^{\oplus d}$, for $U$ a hyperbolic plane and $-E_8$ the
	$E_8$ lattice with the negative of its usual pairing.
%
\end{remark}

We next sketch the idea for proving \autoref{theorem:n-monodromy-image}, whose proof will occupy much of the remainder of the section.

\subsubsection*{Idea of proof of \autoref{theorem:n-monodromy-image}}
First, in the case $k = \bc$, \autoref{theorem:n-monodromy-image} follows from \cite[Theorem 4.10]{de-jongF:on-the-geometry-of-principal-homogeneous-spaces}.
Therefore, the content of \autoref{theorem:n-monodromy-image} is to show it holds over fields $k$ of positive characteristic not equal to $2$.
To transfer the monodromy to positive characteristic, the key issue is showing that $\mono n d {\ol k}$ factors through the tame fundamental group,
meaning that a compactification of the corresponding cover has no ramification orders dividing $\chr(k)$.
We have shown in \autoref{lemma:2-divisor-complement} that there are two divisors in the boundary of $\smespace d {\bz[1/2]}$ 
whose smooth locus over $\bz[1/2]$ is dense in all fibers over $\spec \bz[1/2]$.
By applying the Lefschetz hyperplane theorem, we can replace $\smespace d {\bz[1/2]}$ by an open subscheme $U \subset \bp^1$, and use the above mentioned smoothness
to conclude that the intersection of $L$ with these boundary divisors is \'etale over $\spec \bz[1/2]$.
A version of Abhyankar's lemma then implies that
$\mono n d {\ol k}$ factors through the tame fundamental group, implying that 
$\im \mono n d {\ol k}$ agrees with $\im \mono n d {\bc}$.

\subsection{Analyzing the boundary divisors}
\label{subsection:boundary-divisors}

In this section we introduce a compactification of $\smespace d B$ and determine the divisors in the complement of $\smespace d B$.
By showing these divisors are generically smooth in fibers over $B$, we will be able to conclude
that the generic ramification orders of the Selmer space over these divisors are tame, which will allow us to compare
the monodromy associated to the Selmer space in characteristic $0$ and characteristic $p$.

The main divisor of interest is $\divsing d B$, introduced in \autoref{definition:divii}. 
We next show this is a relative effective Cartier divisor. 
To do so, we will want to understand its fibers over $B$.
This will use the following lemma, concretely connecting it
to $\mathcal C^d_B$ of \autoref{definition:divii} in certain cases.
\begin{lemma}
\label{lemma:divii-equals-c}
	For $\spec k$ a point of $\spec \bz[1/2]$, 
	we have
	$\divsing d k = \mathcal C^d_k$.
\end{lemma}
\begin{proof}
By construction, $\divsing d k$ is the base change of $\divsing d {\bz[1/2]}$ so we only need to show $\mathcal C^d_k$ is the base change of
$\mathcal C^d_{\bz[1/2]}$.
The map $\pi_1$ of \autoref{definition:divii} is proper and $\mathcal C^d_B$ is reduced whenever $B$ is reduced 
	by \autoref{lemma:pi-2-smooth}.
	This implies that $\mathcal C^d_k$, the scheme theoretic image of $\Psi^d_k$ under $\pi_1$,
	is simply the reduced subscheme of $\aff d k$ whose underlying set is given by the set-theoretic image of the map $\pi_1$ from \eqref{equation:incidence-divisor}.
	So indeed, $\mathcal C^d_k$ is the base change of $\mathcal C^d_{\bz[1/2]}$ to $\spec k$, and hence agrees with $\divsing d k$.
\end{proof}

We can now use the above lemmas to show $\divsing d B$ is a relative effective Cartier divisor.
\begin{proposition}
	\label{lemma:divii-geom-integral}
	For $B$ a scheme with $2$ invertible on $B$, the subscheme 
	$\divsing d B \subset \aff d B$
	as in \autoref{definition:divii}
	is 
	a relative effective Cartier divisor over $B$.
	In the case that $B = \spec k$ for $k$ a field,
	$\divsing d B$ is geometrically integral.
\end{proposition}
\begin{proof}
	Since the property of being a relative effective Cartier divisor is preserved under base change, to verify the first statement, it suffices to verify
	the universal case that $B = \spec \bz[1/2]$.
	By the equivalence of various notions of effective relative Cartier divisor \cite[\S8.2, Lemma 6]{BoschLR:Neron},
	it suffices to check $\divsing d B$ is Cartier and is Cartier in each fiber. 
	Since $\aff d {\bz[1/2]}$ is regular, the notions of Cartier divisor and Weil divisor coincide,
	and so it suffices to check that
	$\divsing d {\bz[1/2]}$ is a Weil divisor which is geometrically integral on each fiber over $\bz[1/2]$.

	Let $R$ be either $\bz[1/2]$ or any residue field of $\bz[1/2]$ over a prime ideal of $\bz[1/2]$.
	To check $\divsing d R$ is a Weil divisor in $\aff d R$, we note that by 
\autoref{lemma:divii-equals-c}
	$\divsing d R =  \mathcal C^d_R$, and so we wish to check $\mathcal C^d_R$ is a Weil divisor.
	Since the map $\pi_1$ is quasi-finite, $\dim \mathcal C^d_R = \dim \Psi^d_R + \dim Y_R = 12d + 2 + \dim R$.
	Since schemes of finite type over $\spec \bz$ are catenary, we obtain that $\divsing d R$ is indeed a divisor in $\aff d R$.
	Finally, in the case $R$ is one of the residue fields of $\bz[1/2]$, $\divsing d R$ is geometrically integral because it is the image of the
	scheme $\Psi^d_R$, which is geometrically integral by \autoref{lemma:pi-2-smooth}.
\end{proof}

Using \autoref{lemma:divii-geom-integral}, we verify that $\divsing d {\bz[1/2]}$ has smooth locus meeting all fibers over $\spec \bz[1/2]$.
For the statement of the next corollary,
recall that $\smespace d {\bz[1/2]}$ was constructed as a fiberwise dense open subscheme of an affine space $\aff d {\bz[1/2]}$,
and so it embeds as a fiberwise dense open in $\bp^{12d+3}_{\bz[1/2]}$ via the composition 
$\smespace d {\bz[1/2]} \subset \aff d {\bz[1/2]} \subset \bp^{12d+3}_{\bz[1/2]}$.
	
\begin{corollary}
	\label{lemma:2-divisor-complement}
The complement $\bp^{12d+3}_{\bz[1/2]} - \smespace d {\bz[1/2]}$ has two irreducible components, $\divsing d {\bz[1/2]}$ and $\bp^{12d+3}_{\bz[1/2]} - \aff d {\bz[1/2]}$,
	both of which are relative effective Cartier divisors and remain irreducible over each point of $\spec \bz[1/2]$.
	Further, $\bp^{12d+3}_{\bz[1/2]} - \aff d {\bz[1/2]}$ is smooth over $\spec {\bz[1/2]}$ and
	there is a dense open subscheme $U \subset \divsing d {\bz[1/2]}$
	meeting each fiber over $\spec {\bz[1/2]}$ nontrivially
	such that $U$ is smooth 
	over ${\bz[1/2]}$.
\end{corollary}
\begin{proof}
First, $\bp^{12d+3}_{\bz[1/2]} - \aff d {\bz[1/2]}$ is simply a projective space of dimension $12d+2$ over $\spec {\bz[1/2]}$, and hence certainly a
smooth relative effective Cartier divisor.

	Therefore, it suffices to show that $\aff d {\bz[1/2]} - \smespace d {\bz[1/2]}$ has only one irreducible component which is a relative
	effective Cartier divisor, 
	given by $\divsing d {\bz[1/2]}$, and that $\divsing d {\bz[1/2]}$ possesses a dense open $U$ as in the statement of the corollary.
	The irreducibility and relative dimension statements follow from \autoref{lemma:divii-geom-integral}.

	To conclude, we verify the existence of $U$.
	It suffices to check that $\divsing d {\bz[1/2]}$ is smooth along an open subscheme of each fiber over $B$.
	Because $\divsing d {\bz[1/2]}$ is a relative Cartier divisor by \autoref{lemma:divii-geom-integral}, it is flat over $\spec \bz[1/2]$, and so we only need check it
	that $\divsing d {\bz[1/2]}$ it is generically smooth in each fiber over $\spec \bz[1/2]$.
	This generic smoothness follows from the fact that it is geometrically integral in each fiber over $\spec \bz[1/2]$, as again was shown in
	\autoref{lemma:divii-geom-integral}.
\end{proof}

\subsection{Monodromy comparison}
	\label{subsection:proof-monodromy-tame}

Using \autoref{lemma:2-divisor-complement}, we can establish tameness of the cover
corresponding to $\mgeom n d k$ for $k$ a field of positive characteristic, and deduce that this group
does not depend on the characteristic of $k$, at least when $\chr(k) \nmid 2n$.
The idea of the proof is to check that the $\mgeom n d k$ cover associated to $\sspace n d k$
is tamely ramified over the boundary divisors of \autoref{lemma:2-divisor-complement}. This will allow us to employ the specialization map (which only exists for tame fundamental groups) relating
$\mgeom n d k$ and $\mgeom n d \bq$.

\begin{proposition}
	\label{proposition:monodromy-tame}
	For $n > 0$, 
	suppose $\chr(k) \nmid 2n$ and $d > 0$.
	Then,
	$\mgeom n d k = \mgeom n d \bq$.
\end{proposition}
\begin{proof}
	Note that $\mono n d {\ol k}$ corresponds to a connected
	$\mgeom n d k$ finite \'etale cover of $\smespace d {\overline{k}}$.
	Since geometric connectedness is preserved upon base extension between fields,
	the statement holds when $\chr(k) = 0$. Therefore, for the rest of the proof we assume $\chr(k) =: p$ with $p \nmid 2n$.

	Take $S := \spec \sco_{\bz[1/2n], p}^{\sh}$, the strict henselization of the local ring at $(p) \subset \bz[1/2n]$.
	The surjection $\pi_1(\smespace d {\ol \bq}) \to \im \mono n d {\overline\bq}$ induced by the Selmer space corresponds to a Galois connected finite \'etale
	cover $T_{\overline \bq} \to \smespace d {\ol \bq}$ with Galois group $\im \mono n d {\overline\bq}$.
	We next construct a finite extension $S'$ of $S$ such that 
	$\im \mono n d {S'} = \im \mono n d {\overline\bq}$.
	By writing $\overline \bq$ as the colimit over all finite extensions of $\bq$,
	we may find a finite extension $K$ of $\bq$ and a cover $T_K \to \smespace d K$ so that the base change of this cover to $\overline{\bq}$
	is $T_{\overline \bq} \to \smespace d {\ol \bq}$.
	This yields a surjection $\pi_1(\smespace d K) \twoheadrightarrow \im \mono n d {\ol{\bq}}$.
	By replacing $S$ by a finite ramified extension $S'$, we may assume $\spec K$ factors through $S'$. 
	Now, relabel $S'$ by $S$ and let $\eta$ denote the generic point of this newly constructed scheme $S$.
	We claim that $\im \mono n d {\ol \bq} =  \im \mono n d {S}$.
	Indeed, since the map $\pi_1(\smespace d \eta) \to \pi_1(\smespace d S)$ is surjective due to normality of $\smespace d S$, we obtain 
	$\im \mono n d {\eta} = \im \mono n d {S}$.
	Further, because 
	$\im \mono n d K = \im \mono n d {\ol \bq}$ by construction of $K$,
	the containments $\im \mono n d {\ol \bq} \subset \im \mono n d \eta \subset \im \mono n d K$ implies
	$\im \mono n d \eta = \im \mono n d {\ol \bq}$.

	Since the residue field of $S$ is $\ol{\mathbb F}_p$, to conclude the proof, we claim it is enough to show that $T \times_{S} \spec \ol{\mathbb F}_p$
	is connected.
	Indeed, this will imply that for any field $k$ of characteristic $p$, $T \times_S \spec \ol{k}$ remains connected.
	Since the connected scheme $T \times_S \spec \ol{\mathbb F}_p$ is then the Galois closure of the cover $\sspace n d {\ol k} \to \espace d {\ol k}$,
	we obtain that the order of the monodromy group $\im \mono n d {\ol k}$
	is equal to the order of $\im \mono n d {\ol \bq} = \im \mono n d S$ because the corresponding covers $T$ and $T \times_S \spec \ol{\mathbb F}_p$
	have the same degree. 
	Since the orders agree and we have a containment $\im \mono n d {\ol k} \subset  \im \mono n d S$,
	this containment must in fact be an equality.

	Hence, it remains to verify $T \times_S \spec \ol{\mathbb F}_p$ is connected.
	To do so, we will intersect $\smespace d S$ with a suitable line.
	We next construct this line.
	Let $\smespace d S \hookrightarrow \bp^{12d+3}_S$ denote the open embedding described in \autoref{lemma:2-divisor-complement}.
	We claim that for a general line $L \simeq \bp^1_S$ with $L \subset \bp^{12d+3}_S$, the map
	$\pi_1^{\et}(L_{\ol \bq} \cap \smespace d {\ol \bq}) \ra \pi_1^{\et}(\smespace d {\ol \bq})$ is surjective.
	This follows from the Lefschetz hyperplane theorem applied over the complex numbers
	\cite[Part II, Theorem 1.2]{goresky1988stratified}
	and invariance of the fundamental group of a quasi-projective variety upon base change between
	algebraically closed fields of characteristic $0$
	\cite{landesman:invariance-of-the-fundamental-group} or 
\cite[Expos\'e XIII, Proposition 4.6]{noopsortSGA1Grothendieck1971}.

	Given any $L$ as above, let $U := L \cap \smespace d S$ and let $D := L - U$.
	From the connected finite \'etale cover $T \to \smespace d S$, we obtain a finite \'etale cover
	$T_U := T \times_{\smespace d S} U \to U$ which is connected because for any $L$ as above, 
	$\pi_1^{\et}(L_{\ol \bq} \cap \smespace d {\ol \bq}) \ra \pi_1^{\et}(\smespace d {\ol \bq})$ is surjective.
	Then, $T_U$ corresponds to a surjective map $\pi_1(U) \twoheadrightarrow \pi_1(\smespace d S) \twoheadrightarrow \im \mono n d {\ol{\bq}} \subset \gl(\vsel n d {\ol{\bq}})$.

	For any such $L$, suppose we knew that $T_U \times_S \spec {\ol{\mathbb F}_p}$ is connected.
	Then, we claim $T \times_S \spec \ol{\mathbb F}_p$ must also be connected.
	Indeed, if $T \times_S \spec \ol{\mathbb F}_p$ were disconnected, it is necessarily a disjoint union of two nonempty finite \'etale covers of
	$\espace d {\ol{\mathbb F}_p}$, and so its restriction to $U$ would also be a disjoint union of two nonempty finite \'etale covers,
	hence disconnected.
	Therefore, in order to verify $T \times_S \spec \ol{\mathbb F}_p$ is connected, and hence that $\im \mono n d {\ol k} =  \im \mono n d S$
	it is enough to find a line $L$ as above so that $T_U \times_S \spec {\ol{\mathbb F}_p}$ is connected.

	To prove $T_U \times_S \spec {\ol{\mathbb F}_p}$ is connected, we want to apply 
	\cite[Expos\'e XIII, 2.10]{noopsortSGA1Grothendieck1971}
	in 
	\eqref{equation:tame-diagram} below. For this we will need to verify that
	$D$ is \'etale over $S$ and $T_U \to U$ is tamely ramified.
	In fact, by
	relative Abhyankar's lemma \cite[Expos\'e XIII, Proposition 5.5]{noopsortSGA1Grothendieck1971}
	it suffices to verify that $D$ is a relative effective Cartier divisor which is \'etale over $S$,
	as this then implies the cover $T_U \to U$ corresponding to the monodromy representation $\mono n d {S}$ is tamely ramified.

	Using \autoref{lemma:2-divisor-complement}, $D$ is obtained as the intersection of $L$ with the two divisors 
	$\bp^{12d+3}_S - \aff d S$ and $\divsing d S$. Both of these divisors are smooth along a dense open meeting the special fiber over $S$ by
	\autoref{lemma:2-divisor-complement}. This is the step we use that $2 \nmid \chr (k)$, see \autoref{remark:char-2}.
	Hence, since $L$ was chosen generally, we may arrange by applications of Bertini's theorem
	that $D$ is \'etale over $S$, as we now explain.
	First, by choosing $L$ generically, 
	we may assume by Bertini's theorem that its intersection with $D$ has generic and special fiber which both have degree equal to
	the degree of $\divsing d S$ plus the degree of $\bp^{12d+3}_S - \aff d S$.
	Properness of $L \cap D$ implies $L \cap D$ is finite over $S$, with both its special and generic fiber having the same degree.
	Therefore, $L \cap D$ is flat over $S$.
	By Bertini's theorem, we may arrange that the special fiber $L \cap D$ is \'etale over $\spec \ol{\mathbb F}_p$, and therefore
	we obtain that $L \cap D$ is flat with \'etale special fiber, hence \'etale over $S$, as we wanted to show.
	Therefore, relative Abhyankar's lemma
	\cite[Expos\'e XIII, Proposition 5.5]{noopsortSGA1Grothendieck1971}
	implies $\mono n d {S}$ is tamely ramified.

	We now conclude the proof by showing how tameness of $T_U \to U$ implies 
	$T_U \times_S \spec \overline{\mathbb F}_p$ is connected, and hence $\mgeom n d k= \mgeom n d {\bq}$.
	Since $\pi_1^{\et}(U) \ra \im \mono n d {\ol \bq}$ factors through $\pi_1^{\tame}(U)$
	we obtain a commutative diagram
	\begin{equation}
		\label{equation:tame-diagram}
		\begin{tikzcd} 
			\pi_1^{\tame}(U_{\ol \bq}) \ar{r} \ar{dr} \ar{d}{\speci} & \pi_1^{\et}(U)\ar{r} \ar {d} & \pi_1^{\et}(\smespace d S) \ar{d} \\
			\pi_1^{\tame}(U_{\ol {\mathbb F}_p}) \ar {r} & \pi_1^{\tame}(U) \ar{r} & \gl(\vsel n d k).
	\end{tikzcd}\end{equation}
	The surjective map $\speci : \pi_1^{\tame}(U_{\ol \bq}) \ra \pi_1^{\tame}(U_{\ol {\mathbb F}_p})$ is the {\em specialization map} (see 
	\cite[Th\'eor\'eme 4.4 and Proposition 5.1]{orgogozo2000theoreme}, or alternatively 
	\cite[Expos\'e XIII, 2.10]{noopsortSGA1Grothendieck1971}).
	By construction, the image of $\pi_1^{\tame}(U_{\ol \bq}) \simeq \pi_1^{\et}(U_{\ol \bq})$ in $\gl(\vsel n d {\ol \bq})$ is $\mgeom n d {\bq}$.
	By commutativity of the diagram and surjectivity of $\speci$, it follows $\pi_1^{\tame}(U_{\ol {\mathbb F}_p})$ also has image $\mgeom n d {\bq}$.
	Since the map 
	$\pi_1^{\tame}(U_{\ol {\mathbb F}_p}) \to \mgeom n d {\bq}$
	corresponds to the finite \'etale cover $T_U \times_S \spec \overline{\mathbb F}_p$
	surjectivity of 
	$\pi_1^{\tame}(U_{\ol {\mathbb F}_p}) \twoheadrightarrow \mgeom n d {\bq}$
	implies 
	$T_U \times_S \spec \overline{\mathbb F}_p$ is connected.
\end{proof}

\subsection{Computing the Monodromy and number of components}
	\label{subsection:monodromy-image}

Using \autoref{proposition:monodromy-tame} and the computation of the monodromy over $\bc$ from
\cite[Theorem 4.10]{de-jongF:on-the-geometry-of-principal-homogeneous-spaces},
we can now prove \autoref{theorem:n-monodromy-image}.
Following this, we compute the number of components of the Selmer space.

\begin{proof}[Proof of \autoref{theorem:n-monodromy-image}]
	First, $\pi_1(\smespace d {\ol k}) \hookrightarrow \pi_1(\smespace d k)$ induces an inclusion
	$\im \mono n d {\ol k} \subset \im \mono n d {k}$.
	Observe that the quadratic form $\qsel n d k$ is preserved by the action
	of $\pi_1(\smespace d k)$ as it is induced by a natural cup
	product and Poincar\'e duality, (as described in \cite[p. 784-785]{de-jongF:on-the-geometry-of-principal-homogeneous-spaces} and also \cite[p. 253-254]{artinSD:the-shafarevich-tate-conjecture-for-pencils-of-elliptic-curves},) functorial under the action of $\pi_1(\smespace d k)$.
	It follows that $\im \mono n d {k} \subset \o(\qsel n d k)$.
	It remains to show $r_n(\osp(\qsel \bz d {\bc})) \subset \im \mono n d {\ol k}$.

	By \autoref{proposition:monodromy-tame}	it suffices to calculate $\im \mono n d {\ol k}$ in the case $\ol k = \bc$.
	The result then essentially follows from \cite[Theorem 4.10]{de-jongF:on-the-geometry-of-principal-homogeneous-spaces},
	as we now explain.
	In \cite[Theorem 4.10]{de-jongF:on-the-geometry-of-principal-homogeneous-spaces}, it is implicitly assumed
	$d \geq 2$, as stated in the first line of the proof of \cite[Theorem 4.9]{de-jongF:on-the-geometry-of-principal-homogeneous-spaces}.
	This hypothesis is used in \cite[p. 787, lines 25-28]{de-jongF:on-the-geometry-of-principal-homogeneous-spaces} to ensure
	that a particular related Dynkin diagram ``contains a certain subdiagram with $6$ vertices.''

At this point, it may be helpful for the reader to recall our constructions of the Selmer stack $\sstack n d B$ and moduli stack of minimal Weierstrass models $\estack d B$ from \autoref{definition:selmer-stack}.
	Since $\smsstack n d \bc \ra \smestack d \bc$ is a quotient of $\smsspace n d \bc \ra \smespace d \bc$ by the action of the smooth connected algebraic group $\bg_a^{2d+1} \rtimes\bg_m$,
	the monodromy representation associated to the map on the coarse spaces of $\smsstack n d \bc \ra \smestack d \bc$ has the same image as the representation
	$\mono n d \bc$ associated to 
	$\smsspace n d \bc \ra \smespace d \bc$.
	The previous statement can be verified algebraically, though it is even easier to verify it topologically, which is viable as we are working over $\bc$.
	In \cite[Theorem 4.10]{de-jongF:on-the-geometry-of-principal-homogeneous-spaces}, it is shown that for $d \geq 2$,
	the resulting monodromy map on coarse spaces has image 
	containing 
	$r_n(\osp(\qsel \bz d {\bc}))$.
	Hence the same is true of $\mono n d \bc$.
	Note here that we are using $\osp(\qsel \bz d {\bc})$
and not $\o^*_{+1}(\qsel \bz d {\bc})$ 
(see \autoref{subsubsection:group-notation})
since in the proof of
\cite[Theorem 4.10]{de-jongF:on-the-geometry-of-principal-homogeneous-spaces},
\cite[Theorem 5.4.3]{ebeling:monodromy-groups-isolateda} is applied to the
lattice $\vsel \bz d \bc$, which satisfies the hypotheses of the
beginning of \cite[\S5.3]{ebeling:monodromy-groups-isolateda} for $\varepsilon = -1$.

	Finally, non-degeneracy of $\qsel n d k$ follows from the explicit description of $\qsel n d k$, see \autoref{remark:quadratic-form-determination}.
\end{proof}

We next record a standard lemma on monodromy actions for completeness.
This will enable us to relate the geometric monodromy of the Selmer space to the number of its irreducible components.
\begin{lemma}
	\label{lemma:number-of-components-via-monodromy}
	Let $U$ be a noetherian integral normal scheme and let
	$\rho: \pi_1(U) \ra \gl(V)$ denote the monodromy representation
	associated to a finite \'etale cover $\pi: X \ra U$ representing a sheaf of free $\bz/n\bz$ modules.
	The irreducible components 
	of
	$X$ can be bijectively identified with
	orbits of
	$\rho$ on $V$.
\end{lemma}
\begin{proof}
	Let $\eta$ denote the generic point of $U$ and 
	let $X_{\eta}$ the generic fiber of $\pi$.
	Observe that $X_{\eta}$ has $\deg \pi$ distinct geometric points.
	Two such geometric points lie in the same irreducible component if and only if
	there is some element of $\pi_1(\eta)$ taking one to the other.
	Therefore, the set of irreducible components of $X$ is identified with
	orbits of the action of $\pi_1(\eta)$. 
	These orbits are in turn identified
	with orbits of
	the action of $\rho$ via the map $\pi_1(\eta) \twoheadrightarrow \pi_1(U)$,
	which is surjective because $U$ is integral and normal \cite[Expos\'e V, Proposition 8.2]{noopsortSGA1Grothendieck1971}.
	Hence, the number irreducible components of $X$ dominating
	$U$
	is the same as the number of orbits of 
	$\rho$ on $V$.
	Finally, because $X \ra U$ is \'etale, all irreducible components dominate $U$,
	so the number of irreducible components of $X$ is the number of orbits of $\rho$ on $V$.
\end{proof}

We now determine the number of irreducible components of the Selmer space.
\begin{corollary}
	\label{corollary:number-components}
	For $k$ a field of characteristic prime to $2n$ and $d \geq 2$, the cover $\sspace n d k \ra \espace d k$ has $\sum_{m \mid n} m$ irreducible components, all of which 
	are geometrically irreducible and dominate $\espace d k$.
\end{corollary}
\begin{proof}
	Because the formation of $\sspace n d k$ is compatible with base change on the field $k$, 
	it suffices to show that for any field $k$, $\sspace n d k \ra \espace d k$ has $\sum_{m \mid n} m$ irreducible components, all of which dominate $\espace d k$. 
	Indeed, having $\sum_{m \mid n} m$ irreducible components over both $k$ and $\overline{k}$ implies all irreducible components must be {\em geometrically} irreducible.
	Because $\sspace n d k \ra \espace d k$ is \'etale, all irreducible components dominate $\espace d k$.

	It remains to show $\sspace n d k$ has $\sum_{m \mid n} m$ irreducible components.
	The irreducible components of $\sspace n d k$ are the identified with the orbits of
	$\im \mono n d k \subset \gl(\vsel n d k)$ on $\vsel n d k$ by \autoref{lemma:number-of-components-via-monodromy}.
	Recall that a vector $\alpha$ in a free $\bz/n\bz$ module $A$ is {\em primitive} if $\alpha$ cannot be written in the form $m\alpha'$ with $m$ non-invertible in $\bz/n\bz$.
	It is shown in \cite[Lemma 4.12]{de-jongF:on-the-geometry-of-principal-homogeneous-spaces} and its proof that for each $i \in \bz/n\bz$,
	the set of primitive vectors $v \in \vsel n d k$ with $\qsel n d k(v) = i$ form a single orbit under the action of both $r_n(\osp(\qsel \bz d k))$ and $\o(\qsel n d k)$, for $r_n$ as in the statement of \autoref{theorem:n-monodromy-image}.
	Therefore, by \autoref{theorem:n-monodromy-image}, the set of primitive vectors form a single orbit under the action of $\mono n d {\ol{k}}$.
	Hence, there are $n$ orbits of $\mgeom n d k$ corresponding to primitive vectors.
	For every non-primitive vector $v \in \vsel n d k$, there is a unique $t\mid n$ so that $v = t v'$ for $v'$ primitive.
	We now partition $\vsel n d k$ by this value of $t$.
	For each $t \mid n$, we can identify the action of $\mgeom n d k$ on $t \cdot (\vsel n d k)$ with the action of $\mgeom {n/t} d k$ on $\vsel {n/t} d k$. 
	As shown above this action
	has $n/t$ orbits whose union is the set of primitive vectors. 
	We therefore have that the total number of orbits is $\sum_{t \mid n} \frac{n}{t} = \sum_{m \mid n} m$.
\end{proof}

\section{Completing the proof of \autoref{theorem:number-components}}
\label{section:finishing-proof}
We are nearly ready to prove our main theorem, \autoref{theorem:number-components}, and the proof is completed in \autoref{subsection:proof-of-main-theorem}, at the end of this section.
To start, due to the fact that $\sspace n d B$ is only an algebraic space and not a scheme, we need the following
slight generalization of the Lang-Weil estimate.
\begin{lemma}
	\label{lemma:algebraic-space-lang-weil}
	Suppose $X$ is a algebraic space of finite type over $\bz$.
	Let $q$ be a prime power so that $\dim X_{\mathbb F_q} = d$.
	Then, if $X_{\mathbb F_q}$ is geometrically integral, $\#X(\mathbb F_{q}) = q^d + O_X(q^{d-1/2})$.
\end{lemma}
\begin{proof}
	Although $X$ may not be a scheme, it has a dense open subspace which is a scheme by \cite[Theorem 6.4.1]{olsson2016algebraic}.
	To apply \cite[Theorem 6.4.1]{olsson2016algebraic}, we are using that $X$ is quasi-separated because it is finite type over $\bz$.
	Therefore, via noetherian induction, we can write $X$ as a finite union of locally closed subspaces which are schemes, so that every
	$\mathbb F_q$ point of $X$ factors through one of these locally closed subspaces.
	Then, the claim follows from 
	\cite[Theorem 7.7.1(ii)]{poonen:rational-points-on-varieties}
	applied to each of these locally closed subspaces.
\end{proof}
We next clarify our convention on counting points of stacks, and state a lemma which will enable us to relate
the point counts for the stack $\sstack n d B$ and the algebraic space $\sspace n d B$.
\begin{definition}
	\label{definition:}
	Let $\scx$ an algebraic stack of finite type over $\bz$.
	For $x \in \scx$, let $\aut_x := x \times_{\scx} x$ denote the automorphism
	group scheme of $\scx$ at $x$.
		Then, define $\# \scx(\mathbb F_q) := \sum_{x \in \scx(\mathbb F_q)} \frac{1}{\# \aut_x(\mathbb F_q)}$. 
	\end{definition}
\begin{lemma}
	\label{lemma:stack-point-count}
	Let $X$ be a smooth algebraic space of finite type over an open subscheme
	$S \subset \spec \bz$
	and let $G$ be a smooth group scheme over $S$ with geometrically
	connected nonempty fibers, with an action on $X$.
	Let $\scx := \left[ X/G \right]$.
	Then,
	\begin{align*}
		\#X \left( \mathbb F_q \right) = \#G\left( \mathbb F_q \right) \cdot \# \scx(\mathbb F_q).
	\end{align*}
	\end{lemma}
\begin{proof}
	After base changing along $\spec \mathbb F_q \ra \spec \bz$, we can assume
	that $X, G$, and $\scx$ lie over $\spec \mathbb F_q$.
	The claim is then established in 
	\cite[Lemma 2.5.1]{behrend:lefschetz-trace-formula},
	whose proof essentially amounts to the orbit stabilizer theorem.
\end{proof}

Using the above lemmas in conjunction with the relations between points of the Selmer space and sizes of Selmer groups from
\autoref{corollary:equality-selmer-fibers} and \autoref{corollary:uniform-bound-selmer-fibers},
we can relate the number of points of the Selmer space to the number of Selmer elements for elliptic curves.
\begin{proposition}
	\label{proposition:stacky-selmer-count}
Fix $d > 0$.
	\begin{enumerate}
		\item For a fixed prime power $q$ with $\gcd(q,2) =1$,
	\begin{align}
		\label{equation:weierstrass-stacky-count}
		\# \estack d {\bz[1/2]}(\mathbb F_q) &= \sum_{E/\mathbb F_q(t), \,  h(E) = d} \frac{1}{\# \aut(E)}.
 \end{align}
 \item
Letting $q$ range over prime powers with $\gcd(q,2n) = 1$,
 \begin{align}
		\label{equation:selmer-stacky-count}
		\# \sstack n d {\bz[1/2n]}(\mathbb F_q) &= \left( 1 + O_{n,d}(q^{-1/2}) \right)  \sum\limits_{\substack{E/\mathbb F_{q}(t), \, h(E) = d}}\frac{\# \sel_n(E)}{\#\aut(E)} .
\end{align}
	\end{enumerate}
\end{proposition}
\begin{proof}
	We first prove \eqref{equation:weierstrass-stacky-count}.
	For $q$ prime to $2$, we just need to check that elliptic curves over $\mathbb F_q(t)$ of height $d$ are in bijection with $\estack d {\bz[1/2]}(\mathbb F_q)$, so that this bijection
	respects automorphisms group sizes.
	This follows because every elliptic curve over $\mathbb F_q(t)$ can be expressed in terms of a minimal Weierstrass equation $y^2z = x^3 + a_2(s,t)x^2z + a_4(s,t)xz^2 + a_6(s,t)z^3$,
	and two are isomorphic precisely if they are related by the $\bg_a^{2d+1} \rtimes \bg_m$ action of \autoref{definition:selmer-stack}, see
	\cite[4.8 and Lemma 4.9]{de-jong:counting-elliptic-surfaces-over-finite-fields}.

	We next prove \eqref{equation:selmer-stacky-count}.
	Using the identification discussed in the previous paragraph between $\estack d {\bz[1/2]}(\mathbb F_q)$ and elliptic curves over $\mathbb F_q(t)$ of height $d$,
	\autoref{lemma:stack-point-count} applied to $\sstack n d {\bz[1/2n]} = \left[ \sspace n d {\bz[1/2n]}/ \bg_a^{2d+1} \rtimes \bg_m \right]$
	and $\estack d {\bz[1/2n]} = \left[ \espace d {\bz[1/2n]}/ \bg_a^{2d+1} \rtimes \bg_m \right]$
	reduces the issue to showing
	$\# \sspace n d {\bz[1/2n]}(\mathbb F_q) = (1 + O_{n,d}(q^{-1/2})) \sum\limits_{\substack{x \in \espace d {\bz[1/2n]}(\mathbb F_q)}} \# \sel_n(E_x)$.
	Let $\pi: \sspace n d {\bz[1/2n]} \ra \espace d {\bz[1/2n]}$ denote the projection.
	Using \autoref{lemma:algebraic-space-lang-weil} and \autoref{corollary:equality-selmer-fibers}, we have the lower bound 
	\begin{align*}
		\# \sspace n d {\bz[1/2n]}(\mathbb F_q) &= 
		(1 + O_{n,d}(q^{-1/2})) \# \smsspace n d {\bz[1/2n]}(\mathbb F_q) \\
		&\leq
	(1 + O_{n,d}(q^{-1/2}))\sum\limits_{\substack{x \in \espace d {\bz[1/2n]}(\mathbb F_q)}} \# \left( \pi^{-1}(x)(\mathbb F_q) \right) \\
		&=
		(1 + O_{n,d}(q^{-1/2}))\sum\limits_{\substack{x \in \espace d {\bz[1/2n]}(\mathbb F_q)}} \# \sel_n(E_x).
	\end{align*}
	By \autoref{lemma:algebraic-space-lang-weil}, \autoref{corollary:equality-selmer-fibers}, and \autoref{corollary:uniform-bound-selmer-fibers},
	we have the upper bound
\begin{equation}
		\label{equation:upper-bound-selmer}
	\scalebox{.82}{\parbox{.5\linewidth}{
	\begin{align*}
		\# \sspace n d {\bz[1/2n]}(\mathbb F_q) &=	 
	\# \pi^{-1}(\smespace d {\bz[1/2n]})(\mathbb F_q)  +  \# \pi^{-1}(\espace d {\bz[1/2n]} - \smespace d {\bz[1/2n]})(\mathbb F_q)\\	
		&=(1 + O_{n,d}(q^{-1/2})) \left(\# \pi^{-1}(\smespace d {\bz[1/2n]})(\mathbb F_q)    +  n^2 \cdot \# \pi^{-1}(\espace d {\bz[1/2n]} - \smespace d {\bz[1/2n]})(\mathbb F_q) \right) \\
		&=(1 + O_{n,d}(q^{-1/2})) \left( \sum_{x \in \smespace d {\bz[1/2n]}(\mathbb F_q)} \# \pi^{-1}(x)(\mathbb F_q)  + n^2 \cdot \sum_{x \in (\espace d {\bz[1/2n]} - \smespace d {\bz[1/2n]})(\mathbb F_q)} \# \pi^{-1}(x)(\mathbb F_q) \right) \\
		&\geq (1 + O_{n,d}(q^{-1/2})) \left(\sum\limits_{\substack{x \in \espace d {\bz[1/2n]}(\mathbb F_q) }} \# \sel_n(E_x) + 
		n^2 \cdot \sum\limits_{\substack{x \in (\espace d {\bz[1/2n]} - \smespace d {\bz[1/2n]})(\mathbb F_q) }} \frac{\# \sel_n(E_x)}{n^2} \right) \\
		&\geq (1 + O_{n,d}(q^{-1/2}))\left(\sum\limits_{\substack{x \in \espace d {\bz[1/2n]}(\mathbb F_q)}} \# \sel_n(E_x) + 
		\sum\limits_{\substack{x \in (\espace d {\bz[1/2n]} - \smespace d {\bz[1/2n]})(\mathbb F_q)}} \# \sel_n(E_x) \right) \\
		&= 
		(1 + O_{n,d}(q^{-1/2}))\sum\limits_{\substack{x \in \espace d {\bz[1/2n]}(\mathbb F_q)}} \# \sel_n(E_x).
	\end{align*}
}}
\end{equation}
Combining the upper and lower bounds above yields \eqref{equation:selmer-stacky-count}.
\end{proof}

We are finally ready to prove our main theorem, \autoref{theorem:number-components}.
Our strategy is to combine the preceding results in this section to relate the left hand side of \eqref{equation:average-selmer} to the number of geometric components
of $\sspace n d {\mathbb F_q}$, which is $\sum_{m \mid n} m$ by 
\autoref{corollary:number-components}.

\subsection{Proof of \autoref{theorem:number-components}}
\label{subsection:proof-of-main-theorem}
	We first argue one may ignore the contributions from elliptic curves of height $0$.
	We have $\# \espace 0 {\mathbb F_q}(\mathbb F_q) < q^3$ (corresponding to choices of $a_{2i}(s,t) \in \mathbb F_q$ for $i \in \{1,2,3\}$).
	Since height $0$ elliptic curves $E$ have 
	$\fcomp{E} =1$, we obtain $\sce^0 \simeq \sce$ and
	$\sel_n(E) \simeq H^1(\bp^1, \sce[l])\simeq H^1(\bp^1, \sce^0[l])$ by 
	\cite[Proposition 5.4(c)]{cesnavicius:selmer-groups-as-flat-cohomology-groups}.
	Therefore, by \autoref{proposition:selmer-space-and-h1},
each elliptic curve has $n$-Selmer group of size at most $H^0(\bp^1_{\mathbb F_q}, \sce^0[n]) = E[n](\mathbb F_q(t)) \leq n^2.$ 
Hence, the contribution to the numerator and denominator in the definition of $\average^{\leq d}(\sel_n/ \mathbb F_q(t))$ from \eqref{equation:average} coming from height
	$0$ curves is at most $n^2 \cdot q^3$. We can safely ignore this contribution in the large $q$ limit for $d > 0$ because the number of isomorphism classes of elliptic curves of height $d$ over $\mathbb F_q(t)$ is on the order of
	$2 \cdot q^{10d+1}$ by \cite[Proposition 4.16]{de-jong:counting-elliptic-surfaces-over-finite-fields}.

	We next claim that the contribution to both the numerator and denominator 
	in the definition of $\average^{\leq d}(\sel_n/ \mathbb F_q(t))$ from \eqref{equation:average} coming from elliptic curves of height $< d$ and
	the closed locus of elliptic curves with more than $2$ automorphisms be ignored in the large $q$ limit.
	Indeed, these contributions can be bounded by applying \autoref{proposition:stacky-selmer-count} to relate them to points of $\estack d {\bz[1/2]}$ and $\sstack n d {\bz[1/2n]}$,
	applying \autoref{lemma:stack-point-count} to relate them to points of $\espace d {\bz[1/2]}$ and $\sspace n d {\bz[1/2n]}$, and \autoref{lemma:algebraic-space-lang-weil} to
	bound the resulting contribution.
	Therefore, we obtain that
\begin{align}
	\nonumber
	\lim_{\substack{q \ra \infty \\ \gcd(q,2n) = 1}}\frac{\sum\limits_{\substack{E/\mathbb F_{q}(t), \, h(E) \leq d}} \# \sel_n(E)}{\# \left\{ E: E/\mathbb F_q(t), \,  h(E) \leq d \right\}}
	&= 
	\lim_{\substack{q \ra \infty \\ \gcd(q,2n) = 1}}\frac{\sum\limits_{\substack{E/\mathbb F_{q}(t)\\ h(E) = d}} \frac{\# \sel_n(E)}{2}}{\sum\limits_{\substack{E/\mathbb F_q(t)\\  h(E) = d}} \frac{1}{2}} \\
	&= 
	\label{equation:add-auts}
\lim_{\substack{q \ra \infty \\ \gcd(q,2n) = 1}}\frac{\sum\limits_{\substack{E/\mathbb F_{q}(t)\\ h(E) = d}} \frac{\# \sel_n(E)}{\#\aut(E)}}{ \sum\limits_{\substack{E/\mathbb F_q(t)\\ h(E) = d}} \frac{1}{\# \aut(E)}}.
\end{align}

By applying \autoref{proposition:stacky-selmer-count} and \autoref{lemma:stack-point-count} to both the numerator and denominator of \eqref{equation:add-auts}, we see
\begin{align}
	\nonumber
	\lim_{\substack{q \ra \infty \\ \gcd(q,2n) = 1}}\frac{\sum\limits_{\substack{E/\mathbb F_{q}(t), \, h(E) = d}} \frac{\# \sel_n(E)}{\#\aut(E)}}{ \sum\limits_{\substack{E/\mathbb F_q(t)\\  h(E) = d}} \frac{1}{\#\aut(E)}} &= \lim_{\substack{q \ra \infty \\ \gcd(q,2n) = 1}}\frac{\# \sstack n d {\bz[1/2n]}(\mathbb F_q)}{\#\estack d {\bz[1/2]}(\mathbb F_q)} \\
	\label{equation:component-ratio}
	&= \lim_{\substack{q \ra \infty \\ \gcd(q,2n) = 1}}\frac{\# \sspace n d {\bz[1/2n]}(\mathbb F_q)}{\#\espace d {\bz[1/2]}(\mathbb F_q)}.
\end{align}
Note also that each irreducible component of $\sspace n d {\bz[1/2n]}$ has geometrically irreducible fiber over $\spec \mathbb F_p \to \spec \bz[1/2n]$:
indeed, this follows from \autoref{corollary:number-components} and \autoref{proposition:monodromy-tame}, the latter implying that the reduction modulo $p$ of any component with geometrically irreducible generic fiber is again a geometrically irreducible.
Hence, using this, the fact that
$\espace d {\bz[1/2]}$ has a single irreducible
component which is geometrically irreducible,
and \autoref{lemma:algebraic-space-lang-weil}, the ratio
in \eqref{equation:component-ratio} is simply the number of irreducible components of $\sspace n d {\bz[1/2n]}$.
This number of components is $\sum_{m \mid n} m$ for $d \geq 2$ by \autoref{corollary:number-components}.
\qed

\bibliographystyle{alpha}
\bibliography{/home/aaron/Dropbox/master}

\newcommand{\etalchar}[1]{$^{#1}$}
\def\cprime{$'$} \providecommand{\noopsort}[1]{}
\begin{thebibliography}{{\noopsort{SGA1}}R71}

\bibitem[Ach06]{Achter:distributionClassGroups}
Jeffrey~D. Achter.
\newblock The distribution of class groups of function fields.
\newblock {\em J. Pure Appl. Algebra}, 204(2):316--333, 2006.

\bibitem[Ach08]{Achter:cohenQuadratic}
Jeffrey~D. Achter.
\newblock Results of {C}ohen-{L}enstra type for quadratic function fields.
\newblock In {\em Computational arithmetic geometry}, volume 463 of {\em
  Contemp. Math.}, pages 1--7. Amer. Math. Soc., Providence, RI, 2008.

\bibitem[ASD73]{artinSD:the-shafarevich-tate-conjecture-for-pencils-of-elliptic-curves}
M.~Artin and H.~P.~F. Swinnerton-Dyer.
\newblock The {S}hafarevich-{T}ate conjecture for pencils of elliptic curves on
  {$K3$} surfaces.
\newblock {\em Invent. Math.}, 20:249--266, 1973.

\bibitem[BCP97]{Magma}
Wieb Bosma, John Cannon, and Catherine Playoust.
\newblock The {M}agma algebra system. {I}. {T}he user language.
\newblock {\em J. Symbolic Comput.}, 24(3-4):235--265, 1997.
\newblock Computational algebra and number theory (London, 1993).

\bibitem[BD09]{behrendD:connected-components-moduli}
Kai Behrend and Ajneet Dhillon.
\newblock Connected {{Components}} of {{Moduli Stacks}} of {{Torsors}} via
  {{Tamagawa Numbers}}.
\newblock {\em Canadian Journal of Mathematics}, 61(01):3--28, February 2009.

\bibitem[Beh93]{behrend:lefschetz-trace-formula}
Kai~A. Behrend.
\newblock The {{Lefschetz}} trace formula for algebraic stacks.
\newblock {\em Inventiones Mathematicae}, 112(1):127--149, December 1993.

\bibitem[BKL{\etalchar{+}}15]{bhargavaKLPR:modeling-the-distribution-of-ranks-selmer-groups}
Manjul Bhargava, Daniel~M. Kane, Hendrik~W. Lenstra, Jr., Bjorn Poonen, and
  Eric Rains.
\newblock Modeling the distribution of ranks, {S}elmer groups, and
  {S}hafarevich-{T}ate groups of elliptic curves.
\newblock {\em Camb. J. Math.}, 3(3):275--321, 2015.

\bibitem[BLR90]{BoschLR:Neron}
Siegfried Bosch, Werner L{\"u}tkebohmert, and Michel Raynaud.
\newblock {\em N\'eron models}, volume~21 of {\em Ergebnisse der Mathematik und
  ihrer Grenzgebiete (3) [Results in Mathematics and Related Areas (3)]}.
\newblock Springer-Verlag, Berlin, 1990.

\bibitem[BS13a]{bhargavaS:average-4-selmer}
Manjul Bhargava and Arul Shankar.
\newblock The average number of elements in the 4-selmer groups of elliptic
  curves is 7.
\newblock {\em arXiv preprint arXiv:1312.7333v1}, 2013.

\bibitem[BS13b]{bhargavaS:average-5-selmer}
Manjul Bhargava and Arul Shankar.
\newblock The average size of the 5-selmer group of elliptic curves is 6, and
  the average rank is less than 1.
\newblock {\em arXiv preprint arXiv:1312.7859v1}, 2013.

\bibitem[BS15a]{bhargava-shankar:binary-quartic-forms-having-bounded-invariants}
Manjul Bhargava and Arul Shankar.
\newblock Binary quartic forms having bounded invariants, and the boundedness
  of the average rank of elliptic curves.
\newblock {\em Ann. of Math. (2)}, 181(1):191--242, 2015.

\bibitem[BS15b]{bhargavaS:ternary}
Manjul Bhargava and Arul Shankar.
\newblock Ternary cubic forms having bounded invariants, and the existence of a
  positive proportion of elliptic curves having rank 0.
\newblock {\em Ann. of Math. (2)}, 181(2):587--621, 2015.

\bibitem[CEF14]{churchEF:representation-stability-finite-fields}
Thomas Church, Jordan~S. Ellenberg, and Benson Farb.
\newblock Representation stability in cohomology and asymptotics for families
  of varieties over finite fields.
\newblock In {\em Algebraic topology: applications and new directions}, volume
  620 of {\em Contemp. Math.}, pages 1--54. Amer. Math. Soc., Providence, RI,
  2014.

\bibitem[Ces16]{cesnavicius:selmer-groups-as-flat-cohomology-groups}
Kestutis Cesnavicius.
\newblock Selmer groups as flat cohomology groups.
\newblock {\em J. Ramanujan Math. Soc.}, 31(1):31--61, 2016.

\bibitem[CP80]{coxP:torsion-in-elliptic-curves}
David~A. Cox and Walter~R. Parry.
\newblock Torsion in elliptic curves over {$k(t)$}.
\newblock {\em Compositio Math.}, 41(3):337--354, 1980.

\bibitem[CTSSD98]{colliot-theleneSSD:hasse-principle-for-pencils-of-curves}
J.-L. Colliot-Th\'el\`ene, A.~N. Skorobogatov, and Peter Swinnerton-Dyer.
\newblock Hasse principle for pencils of curves of genus one whose {J}acobians
  have rational {$2$}-division points.
\newblock {\em Invent. Math.}, 134(3):579--650, 1998.

\bibitem[Dao17]{dao2017average}
Van~Thinh Dao.
\newblock Average size of 2-selmer groups of jacobians of hyperelliptic curves
  over function fields.
\newblock {\em arXiv preprint arXiv:1711.06147v1}, 2017.

\bibitem[dJ02]{de-jong:counting-elliptic-surfaces-over-finite-fields}
A.~J. de~Jong.
\newblock Counting elliptic surfaces over finite fields.
\newblock {\em Mosc. Math. J.}, 2(2):281--311, 2002.
\newblock Dedicated to Yuri I. Manin on the occasion of his 65th birthday.

\bibitem[dJF11]{de-jongF:on-the-geometry-of-principal-homogeneous-spaces}
A.~J. de~Jong and Robert Friedman.
\newblock On the geometry of principal homogeneous spaces.
\newblock {\em Amer. J. Math.}, 133(3):753--796, 2011.

\bibitem[DR73]{deligneR:les-schemas-de-modules-de-courbes-elliptiques}
P.~Deligne and M.~Rapoport.
\newblock Les sch\'emas de modules de courbes elliptiques.
\newblock pages 143--316. Lecture Notes in Math., Vol. 349, 1973.

\bibitem[Ebe87]{ebeling:monodromy-groups-isolateda}
Wolfgang Ebeling.
\newblock {\em The {{Monodromy Groups}} of {{Isolated Singularities}} of
  {{Complete Intersections}}}, volume 1293 of {\em Lecture {{Notes}} in
  {{Mathematics}}}.
\newblock {Springer Berlin Heidelberg}, Berlin, Heidelberg, 1987.

\bibitem[EVW16]{EllenbergVW:cohenLenstra}
Jordan~S. Ellenberg, Akshay Venkatesh, and Craig Westerland.
\newblock Homological stability for {H}urwitz spaces and the {C}ohen-{L}enstra
  conjecture over function fields.
\newblock {\em Ann. of Math. (2)}, 183(3):729--786, 2016.

\bibitem[FLR20]{fengLR:geometric-distribution-of-selmer-groups}
Tony Feng, Aaron Landesman, and Eric Rains.
\newblock The geometric distribution of {S}elmer groups of elliptic curves over
  function fields.
\newblock {\em arXiv preprint arXiv:2003.07517v1}, 2020.

\bibitem[FW18a]{farbW:etale-homological-stability-and-arithmetic-statistics}
Benson Farb and Jesse Wolfson.
\newblock \'{E}tale homological stability and arithmetic statistics.
\newblock {\em Q. J. Math.}, 69(3):951--974, 2018.

\bibitem[FW18b]{farb2018resolvent}
Benson Farb and Jesse Wolfson.
\newblock Resolvent degree, hilbert's 13th problem and geometry.
\newblock {\em arXiv preprint arXiv:1803.04063v1}, 2018.

\bibitem[GM88]{goresky1988stratified}
Mark Goresky and Robert MacPherson.
\newblock {\em Stratified {M}orse theory}, volume~14 of {\em Ergebnisse der
  Mathematik und ihrer Grenzgebiete (3) [Results in Mathematics and Related
  Areas (3)]}.
\newblock Springer-Verlag, Berlin, 1988.

\bibitem[Gre10]{Greicius:surjective}
Aaron Greicius.
\newblock Elliptic curves with surjective adelic {G}alois representations.
\newblock {\em Experiment. Math.}, 19(4):495--507, 2010.

\bibitem[Gro68]{grothendieck:brauer-iii}
Alexander Grothendieck.
\newblock Le groupe de {B}rauer. {III}. {E}xemples et compl\'ements.
\newblock In {\em Dix expos\'es sur la cohomologie des sch\'emas}, volume~3 of
  {\em Adv. Stud. Pure Math.}, pages 88--188. North-Holland, Amsterdam, 1968.

\bibitem[Hal08]{hall:bigMonodromySympletic}
Chris Hall.
\newblock Big symplectic or orthogonal monodromy modulo {$l$}.
\newblock {\em Duke Math. J.}, 141(1):179--203, 2008.

\bibitem[HLHN14]{ho-lehung-ngo:average-size-of-2-selmber-groups}
Q.~P. H{\`{\^o}}, V.~B. L\^e~H\`ung, and B.~C. Ng\^o.
\newblock Average size of 2-{S}elmer groups of elliptic curves over function
  fields.
\newblock {\em Math. Res. Lett.}, 21(6):1305--1339, 2014.

\bibitem[Kat05]{katz:moments-monodromy-and-perversity}
Nicholas~M. Katz.
\newblock {\em Moments, monodromy, and perversity: a {D}iophantine
  perspective}, volume 159 of {\em Annals of Mathematics Studies}.
\newblock Princeton University Press, Princeton, NJ, 2005.

\bibitem[Lan20]{landesman:invariance-of-the-fundamental-group}
Aaron Landesman.
\newblock Invariance of the fundamental group under base change between
  algebraically closed fields.
\newblock {\em arXiv preprint arXiv:2005.09690v1}, 2020.

\bibitem[Liu02]{liu:algebraic-geometry-and-arithmetic-curves}
Qing Liu.
\newblock {\em Algebraic geometry and arithmetic curves}, volume~6 of {\em
  Oxford Graduate Texts in Mathematics}.
\newblock Oxford University Press, Oxford, 2002.
\newblock Translated from the French by Reinie Ern{\'e}, Oxford Science
  Publications.

\bibitem[Mil80]{Milne:etaleBook}
James~S. Milne.
\newblock {\em \'{E}tale cohomology}, volume~33 of {\em Princeton Mathematical
  Series}.
\newblock Princeton University Press, Princeton, N.J., 1980.

\bibitem[Ogg67]{ogg:elliptic-curves-and-wild-ramification}
A.~P. Ogg.
\newblock Elliptic curves and wild ramification.
\newblock {\em Amer. J. Math.}, 89:1--21, 1967.

\bibitem[Ols16]{olsson2016algebraic}
Martin Olsson.
\newblock {\em Algebraic spaces and stacks}, volume~62 of {\em American
  Mathematical Society Colloquium Publications}.
\newblock American Mathematical Society, Providence, RI, 2016.

\bibitem[OV00]{orgogozo2000theoreme}
F.~Orgogozo and I.~Vidal.
\newblock Le th\'eor\`eme de sp\'ecialisation du groupe fondamental.
\newblock In {\em Courbes semi-stables et groupe fondamental en g\'eom\'etrie
  alg\'ebrique ({L}uminy, 1998)}, volume 187 of {\em Progr. Math.}, pages
  169--184. Birkh\"auser, Basel, 2000.

\bibitem[Poo17]{poonen:rational-points-on-varieties}
Bjorn Poonen.
\newblock {\em Rational points on varieties}, volume 186 of {\em Graduate
  Studies in Mathematics}.
\newblock American Mathematical Society, Providence, RI, 2017.

\bibitem[PR12]{poonenR:random-maximal-isotropic-subspaces-and-selmer-groups}
Bjorn Poonen and Eric Rains.
\newblock Random maximal isotropic subspaces and {S}elmer groups.
\newblock {\em J. Amer. Math. Soc.}, 25(1):245--269, 2012.

\bibitem[PW21]{parkW:average-selmer-rank-in-quadratic-twist-families}
Sun~Woo Park and Niudun Wang.
\newblock Average size of selmer group in large q limit.
\newblock {\em arXiv preprint arXiv:2102.00549v2}, 2021.

\bibitem[Ray95]{raynaud:characteristic-deuler-poincare}
Michel Raynaud.
\newblock Caract\'eristique d'{E}uler-{P}oincar\'e d'un faisceau et cohomologie
  des vari\'et\'es ab\'eliennes.
\newblock In {\em S\'eminaire {B}ourbaki, {V}ol. 9}, pages Exp. No. 286,
  129--147. Soc. Math. France, Paris, 1995.

\bibitem[{\noopsort{SGA1}}R71]{noopsortSGA1Grothendieck1971}
A.~{\noopsort{SGA1}}Grothendieck and M.~Raynaud.
\newblock {\em Rev\^etements \'etales et groupe fondamental}.
\newblock Springer-Verlag, Berlin-New York, 1971.
\newblock S{\'e}minaire de G{\'e}om{\'e}trie Alg{\'e}brique du Bois Marie
  1960--1961 (SGA 1).

\bibitem[Sil94]{Silverman:Advanced}
Joseph~H. Silverman.
\newblock Advanced topics in the arithmetic of elliptic curves.
\newblock 151:xiv+525, 1994.

\bibitem[Sil09a]{Silverman2009}
J.~H. Silverman.
\newblock {\em The arithmetic of elliptic curves}, volume 106 of {\em Graduate
  Texts in Mathematics}.
\newblock Springer, Dordrecht, second edition, 2009.

\bibitem[Sil09b]{Silverman:AEC}
Joseph~H. Silverman.
\newblock {\em The arithmetic of elliptic curves}, volume 106 of {\em Graduate
  Texts in Mathematics}.
\newblock Springer, Dordrecht, second edition edition, 2009.

\bibitem[{Sta}]{stacks-project}
The {Stacks Project Authors}.
\newblock {\itshape Stacks Project}.
\newblock \url{http://stacks.math.columbia.edu}.

\bibitem[Vak01]{vakil:twelve-points-on-the-projective-line}
Ravi Vakil.
\newblock Twelve points on the projective line, branched covers, and rational
  elliptic fibrations.
\newblock {\em Math. Ann.}, 320(1):33--54, 2001.

\bibitem[Vas03]{vasiu2003surjectivity}
A.~Vasiu.
\newblock Surjectivity criteria for {$p$}-adic representations. {I}.
\newblock {\em Manuscripta Math.}, 112(3):325--355, 2003.

\end{thebibliography}

\end{document}